\documentclass[oneside,11pt]{article}

\usepackage{amsmath}
\usepackage{amssymb}
\usepackage{pxfonts}
\usepackage{dsfont}
\usepackage{graphicx}
\usepackage{eucal}
\usepackage{mathrsfs}
\usepackage{theorem}
\usepackage{pifont}
 \usepackage{sectsty}
\usepackage{amscd}
\usepackage{color}
\usepackage{fancyhdr}
\usepackage{framed}
\usepackage[all]{xy}
\usepackage[backref=page]{hyperref}
\usepackage{todonotes}
\usepackage{subcaption}
\definecolor{shadecolor}{rgb}{0.8,0.8,0.8}
\usepackage[stable]{footmisc}
\usepackage[normalem]{ulem}

\theoremheaderfont{\fontfamily{pzc}\bfseries\large}
\newtheorem{theorem}{Theorem}[section]

\newtheorem{lemma}[theorem]{Lemma}
\newtheorem{proposition}[theorem]{Proposition}

\newtheorem{corollary}[theorem]{Corollary}
\newtheorem{definition}[theorem]{Definition}
\newtheorem{question}[theorem]{Question}
\newtheorem{remark}[theorem]{Remark}

\theorembodyfont{\rmfamily}

\newcommand{\specexercise}[1]{}

\newenvironment{proof}{{\flushleft \emph{Proof}:}}{\hfill\ding{110}}
\newenvironment{proof2}[1]{{\flushleft \emph{Proof of {#1}}:}}{\hfill\ding{110}}

\newcommand{\g}{\mathfrak{g}}

\newcommand{\Vol}{\text{Vol}}
\newcommand{\dVol}{\textup{d}\text{Vol}}

\newcommand{\M}{\mathcal{M}}
\newcommand{\R}{\mathbb{R}}

\newcommand{\vp}{\varphi}
\newcommand{\dist}{\operatorname{dist}}
\newcommand{\diam}{\operatorname{diam}}
\renewcommand{\div}{\operatorname{div}}
\newcommand{\id}{\operatorname{Id}}
\newcommand{\supp}{\operatorname{supp}}
\newcommand{\Diff}{\operatorname{Diff}}
\newcommand{\Diffc}{\operatorname{Diff}_\text{c}}

\newcommand{\length}{\operatorname{length}}

\newcommand{\e}{\varepsilon}

\newcommand{\w}{\omega}

\newcommand{\pl}{\partial}
\newcommand{\ind}{\mathds{1}}

\newcommand{\beq}{\begin{equation}}
\newcommand{\eeq}{\end{equation}}

\newcommand{\brk}[1]{\left(#1\right)}          
\newcommand{\Brk}[1]{\left[#1\right]}          
\newcommand{\BRK}[1]{\left\{#1\right\}}        
\newcommand{\Abs}[1]{\left|#1\right|}        

\newcommand{\appendixnumberline}[1]{Appendix\space}

\let\oldappendix\appendix
\makeatletter
\renewcommand{\appendix}{%
  \renewcommand{\@seccntformat}[1]{Appendix~\csname the##1\endcsname\quad}%
  \oldappendix
}
\makeatother

\numberwithin{equation}{section}

\begin{document}

\title{Can we run to infinity? The diameter of the diffeomorphism group with respect to right-invariant Sobolev metrics}
\author{Martin Bauer\footnote{Department of Mathematics, Florida State University. Email: bauer@math.fsu.edu }\, and Cy Maor\footnote{Einstein Institute of Mathematics, The Hebrew University of Jerusalem. Email: cy.maor@mail.huji.ac.il}}
\date{}
\maketitle

\begin{abstract}
The group $\Diff(\M)$ of diffeomorphisms of a closed manifold $\M$ is naturally equipped with various right-invariant Sobolev norms $W^{s,p}$.
Recent work showed that for sufficiently weak norms, the geodesic distance collapses completely (namely, when $sp\le \dim\M$ and $s<1$).
But when there is no collapse, what kind of metric space is obtained? In particular, does it have a finite or infinite diameter?
This is the question we study in this paper.
We show that the diameter is infinite for strong enough norms, when $(s-1)p\ge \dim\M$, and that for spheres the diameter is finite when $(s-1)p<1$.
In particular, this gives a full characterization of the diameter of $\Diff(S^1)$.
In addition, we show that for $\Diff_c(\R^n)$, if the diameter is not zero, it is infinite.
\end{abstract}

\tableofcontents
\section{Introduction and main results}

Right-invariant Sobolev metrics on diffeomorphism groups (or on subgroups thereof) arise naturally in several contexts --- they play a central role in mathematical shape analysis, appear in symplectic geometry and their geodesic equations turn out to be related to several important partial differential equations in hydrodynamics (some more details are given in Section~\ref{sec:background_right_inv} below).

The basic setting is the following (see Section~\ref{sec:Sobolev} for details):
The diffeomorphism group $\Diff(\M)$ of a compact manifold $\M$ is a Lie group, whose associated Lie algebra is the space of vector fields $\mathcal X(\M)$. 
One can equip the diffeomorphism group with right-invariant Finsler metrics, by considering  norms on the Lie algebra $\mathcal X(\M)$;
in this article we will focus on Sobolev norms $W^{s,p}$.
For $p=2$, these norms induce Riemannian metrics on the group $\Diff(\M)$; 
these are the most important metrics for hydrodynamics and shape analysis.
The corresponding geodesic distance between $\varphi_0,\varphi_1\in \Diff(\M)$ induced by this metric is then given by the variational problem
\[
\dist_{s,p}(\varphi_0,\varphi_1):=\underset{\substack{\varphi:[0,1]\to \Diff_c(\M)\\ \varphi(0)=\varphi_0,\,\varphi(1)=\varphi_1}}{\operatorname{inf}} \int_0^1 \|\partial_t \varphi \circ \varphi^{-1}\|_{W^{s,p}}.
\]
If the metric is weak enough, then this distance vanishes identically on every connected component; this phenomenon was first shown by Michor and Mumford \cite{michor2005vanishing}, and was then analyzed in a series of works, culminating in a complete characterization in \cite{jerrard2019geodesic}.
In this article we study  a finer property of the geodesic distance, namely, the diameter of (connected components of) diffeomorphism groups with respect to the geodesic distance induced by these metrics:
\[
\diam_{s,p} \Diff(\M) := \underset{\varphi_0,\varphi_1 \in \Diff(\M)}{\operatorname{sup}}\dist_{s,p}(\varphi_0,\varphi_1).
\]
Our main result is the following characterization of boundedness/unboundedness of this diameter: 

\begin{theorem}\label{thm:diam_compactM}
Let  $\M$ be a compact  manifold without boundary of dimension $n\ge 1$.
Then the diameter of the connected component of the identity of the group of smooth diffeomorphisms with respect to the right-invariant $W^{s,p}$-metric is
\begin{enumerate}
\item zero when $s\le \frac{n}{p}$ and $s<1$, i.e.,  $\diam_{s,p} \Diff(\M)=0$; \;
\item for $M=S^n$, bounded but non-zero for $s\in (\frac{n}{p},1+\frac{1}{p})\cup [1,1+\frac{1}{p})$, i.e., $\diam_{s,p} \Diff(S^n) \in (0,\infty)$;
\item unbounded for $s \geq 1+\frac{n}{p}$, i.e.,   $\diam_{s,p} \Diff(\M)=\infty$.
\end{enumerate}
\end{theorem}
Here in $W^{s,p}$ we denote by $s$ the number of derivatives and by $p$ the exponent. The exact definition of these Sobolev norms, in particular for non-integer $s$, appears in Section~\ref{sec:Sobolev}. 
As mentioned above, the first part of this theorem is known from recent previous results on vanishing geodesic distance (see \cite{jerrard2019geodesic} and the references therein); the second and third parts are the main contributions of this paper. 
More generally, we prove the second part of the theorem for any closed manifold $\M$ that satisfies a certain uniform fragmentation property (see Definition~\ref{def:uni_fragment}), which holds in particular for  spheres.

Note that for $\M=S^1$ our theorem gives a complete characterization of boundedness (unboundedness, resp.) of the diameter of $\Diff(S^1)$, i.e., we have that $\diam_{s,p} \Diff(S^1)$ is bounded for $s<1+\frac{1}{p}$ and unbounded otherwise.
For higher dimensional spheres, there is a gap in the range $s\in \left[1+\frac{1}{p}, 1 + \frac{n}{p}\right)$, in which we do not know whether the diameter is finite or not.
We believe (as explained in Section~\ref{sec:disp_energy_M}), that the diameter is finite in this range, that is, that the transition to infinite diameter happens at $s=1 + \frac{n}{p}$.

If $\M$ is the non-compact space $M=\mathbb R^n$, we prove that the diameter only exhibits two different behaviors: it is either zero or unbounded. This leads to following complete characterization:
\begin{theorem}
\label{thm:diam_diff_R_n}
Let $p\geq 1$. The diameter of $\diam_{s,p} \Diffc(\R^n)$ with respect to the right-invariant $W^{s,p}$-metric is infinite if and only if $s\geq 1$ or $sp> n$. 
In any other case the diameter is zero.
\end{theorem}
Here $\Diff_c(\R^n)$ denotes the connected component of the identity of the group of compactly supported diffeomorphisms.
Again, the zero diameter part is due to previous work, and the contribution of this paper is the infinite diameter part.

As described below, results on the boundedness/unboundedness of the diameter have been studied for a long time in the context of metrics on symplectomorphisms and volume-preserving diffeomorphisms.
To the best of our knowledge, this work is the first to address this question for the full diffeomorphism group, and to show a transition of zero to finite to infinite diameter for a hierarchy of metrics.

\subsection{Right-invariant Sobolev metrics on diffeomorphism groups --- where they arise\footnote{This is, by no means, an excessive survey.}}
\label{sec:background_right_inv}
The interest in right-invariant metrics on diffeomorphism groups originates from Arnold's seminal observation~\cite{arnold1966geometrie} that Euler's equation for the motion of an incompressible fluid admits a geometric interpretation in this setup: it is the geodesic equation of the right-invariant $L^2$-metric on the group of all volume preserving diffeomorphisms (we will refer to this group also as volumorphism group).
Subsequent to Arnold's geometric interpretations for Euler's equation similar formulations have been found for several other partial differential equations that are of relevance in the field of mathematical physics; examples include 
\begin{itemize}
\item the Camassa--Holm equation \cite{CH1993,Mis1998,Kou1999}, which corresponds to the $W^{1,2}$-metric on $\Diff(S^1)$; 
\item the Hunter--Saxton equation~\cite{HS1991,KM2003,lenells2007hunter,Len2008}, as the geodesic equation of the homogeneous $W^{1,2}$-metric on $\Diff(S^1)/S^1$; and more generally the $p$-Hunter-Saxton equation, the geodesic equation of the homogeneous $W^{1,p}$-metric on $\Diff(S^1)/S^1$, as introduced recently by Cotter et al.~\cite{Cotter_rHunterSaxton}.

\item the modified Constantin--Lax--Majda equation \cite{CLM1985,Wun2010,EKW2012,BKP2016}  corresponding to the homogeneous $W^{1/2,2}$-metric on the same space. See, e.g., \cite{TV2011} and the references therein for further examples of Euler--Arnold equations, that are of relevance in mathematical physics.
\end{itemize} 

An additional motivation for the study of (higher order) right-invariant metrics on the full diffeomorphism group stems from their central role in the field of mathematical shape analysis, where differences between objects such as point clouds, images, surfaces, or densities are encoded in the spirit of Grenander's pattern analysis~\cite{grenander1993general, mumford2010pattern,grenander1998computational} by the cost of the minimal (diffeomorphic) transformation that (approximately) transports a source shape to a target shape. 
Using a right-invariant metric on the diffeomorphism group to measure the cost of these diffeomorphic transformations yields the so-called LDDMM-setting~\cite{miller2002metrics,beg2005computing,joshi2000landmark,younes2010shapes,bauer2014overview}, which has proven successful in numerous applications in computational anatomy and medical imaging.

In yet another important line of research, right-invariant Sobolev metrics play a role in symplectic and contact geometry, starting from the Hofer metric on Hamiltonian symplectomorphisms \cite{Hof90}, which is in this context a right-invariant $W^{-1,\infty}$ metric (which is also a bi-invariant metric).

\subsection{Previous results on the geometry induced by right-invariant Sobolev metrics}
The geodesic equations of right-invariant metrics, as they are related to many important partial differential equations, have been studied extensively, starting from Ebin and Marsden~\cite{ebin1970groups} who obtained local well-posedness and stability results for solutions to Euler's equation by studying the geodesic spray of the right-invariant $L^2$-metric on volumorphisms.
Subsequently, local well-posedness results (and sometimes even global existence), have been obtained, using analogous methods, for geodesic equations on the diffeomorphism group as well~\cite{misiolek2010fredholm,Mis1998,Kou1999,BEK2015,TY2005}. 
See~\cite{Kol2017,BV2014} for an overview on these results.
Furthermore, Preston et al.~\cite{misiolek2010fredholm,khesin2013curvatures} studied the curvature of the corresponding spaces and showed Fredholm properties of the exponential map  for both volumorphisms and diffeomorphism groups.

In addition to the geodesic equation itself, right-invariant Sobolev metrics enable us to measure the lengths of curves, hence they give a structure of a length space on these diffeomorphism groups.
A natural question is then --- is this structure degenerate? That is, can the distance between two distinct diffeomorphisms be zero (meaning that there are arbitrarily short curves between them)? This is known as the vanishing geodesic distance phenomenon.
On the other hand, one can ask --- can we find two diffeomorphisms that are arbitrarily far away (i.e., the diameter is infinite)? This is the question we address in this paper.
Note that these questions are of importance when this geodesic distance is explicitly used, e.g., as the regularization term in the LDDMM-setting.

The vanishing geodesic distance phenomenon was first shown for Hamiltonian symplectomorphisms \cite{eliashberg1993biinvariant} under $W^{-1,p}$ metrics for $p<\infty$ (in contrast to the Hofer $W^{-1,\infty}$ metric, which is non-degenerate).
This was later extended to stronger metrics in \cite{bauer2018vanishing}.
Similar results were later obtained for contactomorphisms \cite{shelukhin2017hofer}.
In the context of the full diffeomorphism group, the geodesic distance has been first investigated by Michor and Mumford in~\cite{michor2005vanishing}, where they showed the degeneracy (vanishing) of the geodesic distance for the $L^2$-metric and the non-degeneracy for metrics of order $W^{1,2}$ and above. 
These results have been later generalized to fractional $W^{s,p}$-metrics and a complete characterization of vanishing (non-vanishing) geodesic distance for this class of metrics has been obtained~\cite{jerrard2019geodesic,jerrard2018vanishing,bauer2018vanishing,bauer2013geodesic,bauer2013geodesicb}.
The first part of Theorem~\ref{thm:diam_compactM} is essentially this characterization.

The diameter question was initiated by Shnirelman~\cite{shnirel1987geometry,shnirelman1994generalized} who studied the diameter of the volumorphism group $\Diff_{\mu}(\M)$ with respect to the geodesic distance of the $L^2$-metric.
In particular he showed the boundedness of the diameter for contractible manifolds of dimension $\dim \M\geq 3$, and conjectured the unboundedness in the two-dimensional case.\footnote{To be exact, Shnirelman proved the boundedness of the diameter of $\Diff_{\mu}(\M)$ when $\M$ is the three-dimensional cube, but his proof can be modified to show the result for contractible manifolds of dimension $\dim \M\geq 3$, 0
see, e.g., \cite{arnold1999topological,khesin2008geometry}. }
For $\M$ being either a (two-dimensional) surface with boundary or a closed surface of genus $g\geq 2$  this conjecture has been shown to be true by Eliashberg and Ratiu~\cite{eliashberg1991diameter} for any $L^p$ metric, $p\ge 1$.
The case of the torus and the (significantly more complicated) case of $S^2$ were proved in \cite{brandenbursky2017lp}, thus proving infinite diameter for any closed two dimensional surface with respect to the $L^p$ metric.
Eliashberg and Ratiu~\cite{eliashberg1991diameter} also show that some higher dimensional manifolds with non-trivial topology have infinite diameter for this metric.
So far, the analogue of Shnirelman's question regarding boundedness (unboundedness resp.)\ of the diameter of the full diffeomorphism group, has not been investigated. 

\subsection{Main ideas in the proofs}
Interestingly, the techniques used in this paper are completely orthogonal to the ones used for studying the diameter of symplectomorphisms and volumorphisms, in the sense that all our proofs, both of boundedness and unboundedness of diameter, rely on volume change.

For sufficiently strong metrics, such that $W^{s,p}$ embeds in $C^1$ (that is, when $(s-1)p>n$), we show that we can bound from below the geodesic distance $\dist_{s,p}(\id,\vp)$ of a diffeomorphism $\vp$ to the identity by the logarithm of the Jacobian determinant of $\vp$ at any point.
In particular, the distance from the identity to a diffeomorphism with an arbitrarily large volume change at a point is arbitrarily large.
We call this the \emph{supercritical case}.

The \emph{critical case} is the one for which this embedding just fails, namely when $(s-1)p=n$.
Here we extend an idea of Lenells \cite{lenells2007hunter} to construct an isometry from a degenerate $W^{1,q}$-type metric that sees \emph{only} volume changes to the space of smooth functions on $\M$, and from this we obtain a lower bound for the diameter of this metric, which diverges with $q$.
This degenerate $W^{1,q}$-type metric is weaker, for any $q<\infty$, than our critical $W^{s,p}$ metric, and by using the bound above and controlling the Sobolev embedding constants, we obtain the unboundedness of the diameter by letting $q\to \infty$.

In the \emph{subcritical case}, when $(s-1)p<n$, we aim to prove that the diameter is bounded, at least for manifolds with nice fragmentation properties (see Section~\ref{sec:diamM_subcritical}). In these cases the question can be reduced to a local question on the diameter of the diffeomorphism group of the Euclidean ball $\Diffc(B_1(\R^n))$.
We show, by a rescaling argument, that for $(s-1)p<n$ the uniform boundedness of $\dist_{s,p}(\id, \vp)$ for any $\vp\in \Diffc(B_1(\R^n))$ is equivalent to the uniform boundedness of $\dist_{s,p}(\id, \psi_\lambda)$ for a class of diffeomorphisms $\psi_\lambda(x) \approx \lambda x$, as $\lambda \to \infty$.
That is, boundedness of the diameter of an arbitrary, radially symmetric change of volume at a point implies the boundedness of the whole diffeomorphism group.
We then show that when $(s-1)p<1$, $\dist_{s,p}(\id, \psi_\lambda)$ is indeed bounded, and give an indication of the fact that arbitrary changes of volume should be of bounded cost for the whole subcritical case $(s-1)p<n$.

\subsection{Some open questions}
\begin{itemize}
\item As mentioned, for spheres we have a gap in the range $s\in \left[1+\frac{1}{p} , 1+\frac{n}{p}\right)$; in order to bridge it and prove the boundedness of diameter in this range, we need to find a better way to transport the identity to the family of diffeomorphisms $\psi_\lambda$ mentioned above.
\item We do not know whether for other closed manifolds, that do not satisfy our fragmentation assumption, the diameter of $\Diff(\M)$ is finite or not in the subcritical case.
	That is, are there closed manifolds for which the diameter is either zero or infinity?
\item Another open line of work is to extend the analysis to $W^{s,p}$ metrics on volumorphisms and symplectomorphisms, in particular for negative values of $s$.	There it is not known what is the critical case below which the geodesic distance vanishes (it is known to vanish for $s\le -1+\frac{1}{p}$ and  to not vanish for $s\ge 0$, see \cite{bauer2018vanishing}).
	Also, to the best of our knowledge, it is not known if a similar phenomenon as seen here for $\Diff(S^n)$, namely a transition zero$\to$finite$\to$infinite diameter, can occur for symplectomorphisms/volumorphisms on closed manifolds.
\item More generally, it would be interesting to better understand the connections between the metric questions (vanishing geodesic distance, boundedness of diameter) to other geometric properties (having a smooth geodesic spray, Fredholm properties, etc.).
\end{itemize}
\paragraph{The structure of this paper:}
In Section~\ref{sec:Sobolev} we define the (fractional order) Sobolev norms we are considering in this paper, discuss some of their embedding properties, and define the right-invariant metrics they induce on the diffeomorphism group.
In Section~\ref{Sec:diamS1} we discuss the one dimensional case, namely, the full characterization of boundedness/unboundedness of the diameter for $\Diff(S^1)$;
this case already includes most of the key ingredients that are used in the higher dimensional case, which is the content of Section~\ref{sec:diffM_diameter}, in which we complete the proof of Theorem~\ref{thm:diam_compactM}.
Finally, in Section~\ref{sec:diam_R_n}, we discuss the case of $\Diffc(\R^n)$ and prove Theorem~\ref{thm:diam_diff_R_n}.
 
\paragraph{Acknowledgements:}
We would like to thank to Stefan Haller, Philipp Harms, Stephen Preston, Tudor Ratiu and Josef Teichman for various discussions during the work on this paper, and to Meital Maor for her help with the figures.
We are in particular grateful to Kathryn Mann and Tomasz Rybicki for introducing us to the literature on fragmentation and perfectness, and to Bob Jerrard for his continuous and valuable help throughout the work on this project.
This project was initiated during the BIRS workshop "Shape Analysis, Stochastic Geometric Mechanics and Applied Optimal Transport" in December 2018; 
we are grateful to BIRS for their hospitality. M. Bauer was partially supported by NSF-grant 1912037 (collaborative research in connection with NSF-grant 1912030).

\section{Right-invariant $W^{s,p}$-norms on diffeomorphism groups}
\label{sec:Sobolev}

Let $\mathcal N$ be a finite dimensional manifold. We are interested in the connected component of the identity of the group $\Diff_c(\mathcal N)$ of all  compactly supported, smooth diffeomorphisms  on $\mathcal N$, 
where $\mathcal N$ is either a closed manifold, the Euclidean space $\mathbb R^n$ or the $n$-dimensional
ball $B_r(\mathbb R^n)$ of radius $r$ in $\mathbb R^n$. 
In the following, by a slight abuse of notation, we will denote the connected component of the identity by $\Diff_c(\mathcal N)$ as well.

For $\mathcal N=\M$ a closed manifold the requirement of a compact support is redundant,
and we will simply write $\Diff(\M)$. 
The following classical result, see e.g.~\cite{banyaga2013structure},
summarizes the group and manifold structure of this infinite dimensional space:
\begin{theorem}
The space of smooth, compactly supported diffeomorphisms $\Diff_c(\mathcal N)$ is a simple, Fr\'echet Lie-group whose Lie-algebra is the set of compactly supported vector fields
$\mathcal X_c(\mathcal N)=C^{\infty}_c(\mathcal N, T\mathcal N)$.
\end{theorem}
When $\mathcal N$ is compact, then $\Diff_c(\mathcal N)= \Diff(\mathcal N)$ is a Fr\'echet manifold and hence a Fr\'echet Lie-group~\cite{leslie1967};
in the non-compact case, this is no longer true but $\Diff_c(\mathcal N)$ can be modeled on an (LF)-space or a convenient vector space \cite[\S 43]{KM97}.
This subtlety will not be relevant for the subsequent analysis.

In Section~\ref{Sec:diamS1} we will be in addition interested in the homogeneous space of all smooth diffeomorphisms of the circle $S^1$ modulo translations, which we will identify with the set of all diffeomorpisms that fix the point $0\in S^1$, i.e., 
\[\Diff(S^1)/S^1\sim \left\{
\varphi \in \Diff(S^1): \varphi(0)=0
\right\},\]
where we identified the circle $S^1$ with the interval $[0,1]$.

\subsection{Fractional order Sobolev spaces on $\mathbb R^n$}
To introduce the class of right-invariant $W^{s,p}$-norms on the diffeomorphism group we will 
start by introducing the fractional order Sobolev spaces $W^{s,p}(\mathbb R^n)$ for real valued functions on $\mathbb R^n$.
There are several constructions of fractional order Sobolev spaces, which typically coincide in the important Hilbert case $p=2$.
Here we will use the Gagliardo-seminorm (also known as Slobodeckij seminorm) approach, resulting in the so called Sobolev--Slobodeckij spaces.

Let $p\in (1,\infty)$. 
For a function $f: \mathbb R^n \to \mathbb R^d$ and $s=k+\sigma$ with $k\in \mathbb N$ and $\sigma\in (0,1)$ we define the homogeneous $\dot W^{s,p}$-norm using the Gagliardo-seminorm via
\begin{equation}
\|f\|_{\dot W^{s,p}}=
\left(\iint_{\R^n\times \R^n} \frac{|D^kf(x) - D^kf(y)|^p}{|x-y|^{n+\sigma p}}\,dx\,dy \right)^{1/p},
\end{equation}
where $D^kf$ denotes the $k$-th differential of $f$.
We extend this definition to the full $W^{s,p}$-norm by adding the $L^p$-norm of the function, i.e.,
\begin{equation}
\|f\|_{W^{s,p}}=\|f\|_{W^{k,p}} +\|f\|_{\dot W^{s,p}}.
\end{equation} 

The fractional order Sobolev spaces $W^{s,p}(\R^n)$, as defined above, satisfy the Sobolev embedding theorem, i.e.,
$W^{s,p}(\mathbb R^n)$ embeds in $C^{0}(\mathbb R^n)$ iff $sp>n$ (see, e.g., \cite[Theorem~3.7]{behzadan2015multiplication}).
The following lemma deals with exact estimates for embeddings in $L^q$-spaces for the critical case $sp=n$:
\begin{lemma}\label{lem:critical_embedding}
Let $p\in(1,\infty)$. Then there exists $C=C(p,n)$ such that for every $f\in W^{n/p,p}(\R^n)$,
\[
\|f\|_{L^q(\R^n)} \le C\|f\|_{W^{n/p,p}(\R^n)}q^{1-\frac{1}{p}}, \quad \forall q\in [p,\infty).
\]
\end{lemma}

\begin{proof}
For $s=n/p$ being an integer see \cite[Theorem~12.33]{leoni2017first}.
When $s$ is not an integer, then the Gagliardo (Sobolev--Slobodeckij) spaces we consider here are equivalent to the Besov spaces $B^s_{p,p}(\R^n)$, i.e., the interpolation space $(L^p(\R^n),W^{N,p}(\R^n))_{\frac{s}{N},p}$ for $N>s$ (see e.g., 
\cite[Section~2.5.1, Remark~4]{Tri78}, or \cite[Theorem~3.1]{behzadan2015multiplication}.).
The result for $(L^p(\R^n),W^{N,p}(\R^n))_{\frac{s}{N},p}$ is the content of \cite[Theorem~9.1]{peetre1966espaces}, in which a more general statement is shown and where the case treated here corresponds to $p=r$ in the notation of \cite{peetre1966espaces}.
\end{proof}

Finally we state the behavior of the $W^{s,p}$-norm with respect to scalings:
\begin{lemma}\label{lem:scaling}
Let $p\in(1,\infty)$ and let $\lambda\in \mathbb R_{>0}$. For any $f\in W^{s,p}(\R^n)$ let
$f^\lambda$ denote the function
\[f^\lambda(x)=\frac{1}{\lambda}f(\lambda x)\;.\]
We then have
\[
\|f^\lambda\|_{\dot W^{s,p}} = \lambda^{(s-1)-\frac{n}{p}} \|f\|_{\dot W^{s,p}},\qquad 
\|f^\lambda\|_{L^p} = \lambda^{-1-\frac{n}{p}} \|f\|_{L^p}\,.
\]
\end{lemma}
\begin{proof}
This follows  immediately by the chain rule and changing variables in the standard Sobolev norm or the Gagliardo seminorm, depending on $s$.
\end{proof}

\subsection{Fractional order Sobolev norms on Riemannian manifolds}\label{sec:Sobolev_manifolds}
We now introduce the corresponding space of real valued functions $W^{s,p}(\M)$ for $\M$ a (compact) Riemannian manifold. 
Following \cite[Sect.~7.2.1]{Triebel1992} let $B_\epsilon(x)$ denote the ball of radius $\epsilon$ with center $x$. We can then choose a finite cover of $\M$ by balls $B_\epsilon(x_\alpha)$ with $\epsilon$ sufficiently small, such that normal coordinates are defined in the ball $B_\epsilon(x)$, and a partition of unity $\rho_\alpha$, subordinated to this cover. Using this data we define the $W^{s,p}$-norm of a function $f$ on $M$ via
\begin{align*}
\| f \|_{W^{s,p}(M,g)}^2 &= \sum_{\alpha} \| (\rho_\alpha f)\circ \exp_{x_\alpha}
\|^2_{W^{s,p}(\R^n)}
\end{align*}
Changing the cover or the partition of unity leads to equivalent norms, 
see \cite[Theorem 7.2.3]{Triebel1992} and thus this choice does not matter to us, as we are mainly interested in boundedness (unboundedness, resp.) of the diameter,
a property which remains invariant under equivalent norms.
For integer $s$ and $p=2$ we get norms which are equivalent to the Sobolev norms treated
in \cite[Chapter 2]{Eichhorn2007}. 
The norms depend on the choice of the Riemannian metric $g$, though again, different choices of metrics result in equivalent norms and thus are immaterial to this paper.
This dependence is worked out in detail in \cite{Eichhorn2007}.

\subsection{Right-invariant fractional order Sobolev metrics on diffeomorphism groups}
For vector fields we use the trivialization of the tangent bundle 
that is induced by the coordinate charts and define the norm in each coordinate
as above. 
This leads to a well-defined $W^{s,p}$-norm (up to the equivalence discussed above) on the Lie algebra
$\mathcal X_c(\M)$ of (compactly supported) vector fields on $\M$.
This norm can be extended in the usual way to a right-invariant Finsler metric on the whole diffeomorphism group, i.e., 
for $\vp\in \Diff(\M)$ and $h \in T_\vp \Diff(\M)$,
\begin{equation}
F^{s,p}_{\varphi}(h):=\| h\circ\varphi^{-1}\|_{W^{s,p}}\;,
\end{equation}
where in the right-hand side the norm is the $W^{s,p}$ norm on $\mathcal{X}(\M)$.
In the important case $p=2$ this norm is equivalent to the standard $H^s$ norm that is induced from the inner product $\langle .,.\rangle_{H^s}$, and therefore we obtain a right-invariant Riemannian metric
\begin{equation}
G^{s}_{\varphi}(h,k):=\langle h\circ\varphi^{-1},k\circ{\varphi^{-1}}\rangle_{H^s}\;.
\end{equation}
This is mentioned here for the sake of completeness --- the Riemannian structure will not play a special role in this paper.

Equipping the diffeomorphism group with a Finsler metric gives rise to the corresponding geodesic distance, which is defined in the usual way via
\begin{align*}
\dist_{s,p}(\varphi_0,\varphi_1):=\underset{\varphi}{\operatorname{inf}} \int_0^1 F^{s,p}_{\varphi}(\partial_t \varphi) dt\;,
\end{align*}
where the infimum is taken over all paths $\varphi:[0,1]\to \Diff_c(\M)$ with $\varphi(0)=\varphi_0$ and $\varphi(1)=\varphi_1$.
Using this we can define the diameter of the diffeomorphism group with respect to the metric $F^{s,p}$ to be
\begin{equation}
\diam_{s,p} \Diff(\M) := \underset{\varphi_0,\varphi_1 \in \Diff(\M)}{\operatorname{sup}}\dist_{s,p}(\varphi_0,\varphi_1)= \underset{\varphi \in \Diff(\M)}{\operatorname{sup}}\dist_{s,p}(\operatorname{id},\varphi)\;.
\end{equation}
Here the second equality is due to the right invariance of the Finsler metric and thus of the geodesic distance function. Note, that all these definitions remain valid if the compact manifold $\M$ is replaced by the non-compact space $\mathbb R^n$ or a connected open subset thereof (with $\Diff_c(\R^n)$ and $\mathcal{X}_c(\R^n)$ instead of $\Diff(\M)$ and $\mathcal{X}(\M)$).

The study of the geodesic distance --- and thus of the diameter --- is closely related to the study of the displacement energy \cite{eliashberg1993biinvariant}, which in our context is defined as follows:
\begin{definition}\label{def:displacement_energy}
Given a manifold $\M$, the displacement energy of a set $A\subset \M$ with respect to the $W^{s,p}$-metric is
\[
E_{s,p}(A) = \inf\BRK{\dist_{s,p}(\id,\vp) \,:\, \vp \in \Diff(\M), \,\vp(A)\cap A = \emptyset}.
\]
\end{definition}
In fact it turns out that the geodesic distance collapses if and only if there exists an open set of zero displacement energy (see, e.g., \cite{eliashberg1993biinvariant,shelukhin2017hofer,bauer2018vanishing, jerrard2019geodesic}).
This provides an important tool for studying vanishing distance phenomena.
One could hope for a similar relation between  bounded displacement energy and finite diameter, which is the geometric property that we aim to study in this article.
A result of this type, in general settings, appears in Appendix~\ref{sec:displacement_energy}, however its assumptions are too restrictive for our applications to diffeomorphism groups (see Lemma~\ref{counter:left_translations_H1}).
Nevertheless, we do analyze the displacement energy, as it still provides some insight on the diameter.

\section{The diameter of $\Diff(S^1)$}\label{Sec:diamS1}
The aim of this section is to prove the following complete characterization of boundedness (unboundedness, resp.) of the diameter of the diffeomorphism group of the circle $S^1$ with respect to  right-invariant 
$W^{s,p}$-norms:
\begin{theorem}
\label{thm:diam_S_1}
Let $p\in(1,\infty)$. The diameter $\diam_{s,p} \Diff(S^1)$ of the diffeomorphism group of $S^1$ is zero for $s\leq 1/p$, bounded (but non-zero) for $1/p< s < 1+1/p$, and unbounded for $s\ge 1+1/p$. 
\end{theorem}
Note, that for the important special case $p=2$ this shows that the diameter $\diam_{s,2} \Diff(S^1)$ of the $H^s$-metric is zero for $s\leq \frac12$, bounded (but non-zero) for $s<3/2$, and unbounded for $s\geq\frac32$. 
For the case $p=1$, the only change is that the diameter is finite but non-zero for $s=1$ (due to the fact that $W^{1,1}$ embeds in $C^0$ in the one dimensional case).

\begin{proof}
The zero diameter result for $s\le \frac1{p}$ ($s<1$ for $p=1$) follows directly from  the results on vanishing geodesic distance in~\cite{jerrard2018vanishing,jerrard2019geodesic,bauer2018vanishing}.
We will split the proof of the remaining cases in three parts: the subcritical case $s<1+1/p$, see Section~\ref{sec:diams1_subcritical}, the critical case $s=1+1/p$, see Section~\ref{sec:diams1_critical}, and the supercritical case $s>1+1/p$, see 
Section~\ref{sec:diams1_supercritical}.
\end{proof}

In the proof of the critical case in Section~\ref{sec:diams1_critical} we will show in addition that the diameter of the homogeneous $W^{1,p}$-metric is bounded between $p$ and $8p$. 
Before we present this analysis we want to point out an open question concerning the continuity of the diameter (in the parameter $s$):
\begin{question}
\label{qn:continuity_diameter}
Is the diameter continuous in the Sobolev index $s$ at the critical exponents, i.e., do we have $\lim_{s\to 1/p+} \diam_{s,p} \Diff(S^1) = 0$ and $\lim_{s\to (1 + 1/p)-} \diam_{s,p} \Diff(S^1) = \infty$?
\end{question}

\subsection{The supercritical case $s>1+1/p$.}\label{sec:diams1_supercritical}
The unboundedness for the supercritical case $s\geq1+1/p$ also follows directly from the analysis for the critical case, which is treated in Section~\ref{sec:diams1_critical}. 
In the following we will present a more elementary proof 
that in addition contains an explicit bound for the geodesic distance and will be of importance in the characterization of the displacement energy in Section~\ref{sec:discplacementenergy_diffs1}. 
\begin{lemma}
Let  $s>1+1/p$. Then the geodesic distance of the right-invariant $W^{s,p}$-norm on $\Diff(S^1)$ satisfies
\beq
\label{eq:lower_bound_derivative_dist}
\log \vp'(x) \leq C \dist_{s,p} (\id,\vp)\;,
\eeq
where $C= C(s,p)$ depends on $s$ and $p$. It follows that 
$\diam_{s,p} \Diff(S^1)$ is unbounded.
\end{lemma}

\begin{proof}
Note that in this regime we have the Sobolev embedding $W^{s,p}(S^1) \subset C^1(S^1)$.
Let $\vp_t$ be any curve starting at $\id$ and ending at $\vp$, and let  $u_t$ be the associated vector field, that is $\pl_t\vp_t = u_t\circ \vp_t$.
Denote $\psi_t = \pl_x\vp_t$.
We then have $\pl_t \psi_t = \pl_x u_t\circ \vp_t \cdot \psi_t$, or in other words, $\pl_t (\log \psi_t)(x) = \pl_x u_t(\vp_t(x))$.
Integrating this, and using the fact that $\log \psi_0 = 0$, we have for any $(s-1)p>1$ and any $x\in S^1$,
\beq
\label{eq:lower_bound_derivative}
\log \vp'(x) = \int_0^1 \pl_t (\log \psi_t(x)) \,dt \le \int_0^1 \|\pl_x u_t\|_{L^\infty} \,dt \le C \int_0^1 \|u_t\|_{W^{s,p}} \,dt,
\eeq
where in the last inequality we used the above-mentioned Sobolev embedding.
Since the above inequality holds for all paths $\vp_t$ connecting the idendity to $\vp$ this yields equation~\eqref{eq:lower_bound_derivative_dist}.
By choosing $\vp$ with $\vp'$ arbitrarily large at a point, we get an arbitrarily large lower bound to the diameter of $\Diff(S^1)$ and thus the unboundedness follows.
\end{proof}

\begin{remark}
For $p=1$, the claim and its proof holds also for $s=2$, since $W^{2,1}(S^1)$ embeds into $C^1(S^1)$.
\end{remark}

\subsection{The critical case $s=1+1/p$.}\label{sec:diams1_critical}
In the critical case, we do not have $W^{s,p}(S^1) \subset C^1(S^1)$ as before (unless $p=1$), however we do have $W^{s,p}(S^1) \subset W^{1,q}(S^1)$ for $q<\infty$. 
Inspired by this, the proof for the critical case consists of two steps --- first, we give a lower bound for the diameter with respect to the $\dot W^{1,q}$-metric, which blows up as $q\to \infty$; second, we use Lemma~\ref{lem:critical_embedding} to show that the diameter bound blows up faster than the embedding constants of $W^{s,p}(S^1) \subset W^{1,q}(S^1)$, hence the diameter with respect to $W^{s,p}$ is infinite.

The following lower bound for the diameter with respect to the $\dot W^{1,q}$-metric is based on a generalization of a result of Lenells~\cite{lenells2007hunter}, where he constructed an explicit solution formula for geodesics of the 
homogeneous $\dot W^{1,2}$-metric. In the following lemma we will extend his construction to all homogeneous $\dot W^{1,q}$-norms with $q\geq 1$:
\begin{lemma}\label{lem:Wsp_sphere}
Let
\begin{equation}
\Phi: \operatorname{Diff}(S^1)/S^1 \to C^{\infty}(S^1,\mathbb R), \qquad \varphi \mapsto q (\varphi')^{1/q}. 
\end{equation}
We have:
\begin{enumerate}
\item The mapping $\Phi$  is an isometric embedding, where $\Diff(S^1)/S^1$ is equipped with the right-invariant homogeneous $\dot W^{1,q}$-norm and 
$C^{\infty}(S^1,\mathbb R)$ with the standard (i.e., non-invariant) $L^q$-norm. 
\item The image $\Phi(\operatorname{Diff}(S^1))\subset C^{\infty}(S^1,\mathbb R)$ is 
an open subset of the $L^q$-sphere of radius $q$ 
given by
\[
\mathcal{U}_q := \BRK{ f\in C^{\infty}(S^1;\R) \,:\, f>0, \|f\|_{L^q} =q}. 
\]
\item For fixed $q\geq 1$ the diameter of the set $\mathcal{U}_q$ is bounded from above and below by 
\beq q < \diam \mathcal{U}_q \leq 8q\;.\eeq
As a consequence $\mathcal{U}_q$ is unbounded for $q\to \infty$.
\end{enumerate}
\end{lemma}
\begin{proof}
The flat $L^q$ metric on $C^{\infty}(S^1,\mathbb R_{>0})$ is given by  
\begin{equation}
| \delta f |_f = \left(\int_{0}^{{1}} |\delta f (\theta)|^q \; d\theta\right)^{1/q}, \quad \delta f \in T_f C^{\infty}(S^1,\mathbb R_{>0}) \cong C^\infty(S^1).
\end{equation}
To see that the mapping 
$\Phi$  is an isometric embedding (where $\operatorname{Diff}(S^1)/S^1$ is equipped with the right-invariant 
$\dot W^{1,q}$-metric) we need to calculate the derivative of $\Phi$. 
We have:
\begin{equation}
d \Phi(\varphi).h = (\varphi')^{1/q-1} h'
\end{equation}
and thus
\begin{equation}
|d \Phi(\varphi).h| = \left(\int_{0}^{{1}}  (\varphi'(\theta))^{1-q} |h'(\theta)|^q d \theta\right)^{1/q}
\end{equation}
which equals exactly the right-invariant, homogeneous $\dot W^{1,q}$-metric. 
The characterization of the image of $\Phi$ follows directly from the definition of $\Diff(S^1)/S^1$.

To calculate the lower bound for the diameter of $\mathcal{U}_q$ we consider the functions $f=q$ and {$g=cq(n\ind_{(0, n^{-q})} + \e\ind_{( n^{-q},1)})$} 
for $\e\ll 1$ and $c\approx 1$ such that $\|g\|_{L^q} = {q}$. Then
\[
\diam \mathcal{U}_q \ge \dist_{\mathcal{U}_q}(f,g) \ge \|f-g\|_{L^q} \approx {2^{1/q}q}, 
\]
where $\diam \mathcal{U}_q$ and $\dist_{\mathcal{U}_q}$ refer to the intrinsic distance in $\mathcal{U}_q$. 

It remains to show that $\diam \mathcal{U}_q$ is bounded from above for each fixed $q$. 
Towards this aim we will construct paths that connect given elements $f,g\in  \mathcal{U}_q$ and are bounded independently of $f$ and $g$. 
Let 
\begin{equation}
f_t(\theta) = \frac{ {q} }{\| \tilde f_t(\cdot) \|_{L^q}} \tilde f_t(\theta)
\quad\text{ with }\quad
\tilde f_t(\theta) = (1-t)f(\theta)+tg(\theta)\;.
\end{equation}
It is easy to see that $f_t(\cdot) \in  \mathcal{U}_q$ for any $t\in[0,1]$. It remains 
to bound the $L^q$-length of $f_t(\theta)$. We have
\begin{equation}
\partial_t f_t(\theta) =  \frac{ q }{\| \tilde f_t(\cdot) \|_{L^q}} \partial_t \tilde f_t(\theta)- \frac{ q}{\| \tilde f_t(\cdot) \|^2_{L^q}} \partial_t \| \tilde f_t(\cdot) \|_{L^q}\, \tilde f_t(\theta) ,
\end{equation}
and thus
\begin{equation}
\dist_{\mathcal{U}_q}(f,g)\leq \int_0^1 \| \partial_t f_t(\cdot)\|_{L^q} dt\leq 
q \left(\int_0^1   \frac{ \| \partial_t \tilde f_t(\cdot) \|_{L^q} }{\| \tilde f_t(\cdot) \|_{L^q}}  dt
+ \int_0^1 \frac{\left|\partial_t \| \tilde f_t(\cdot) \|_{L^q}\right|}{\| \tilde f_t(\cdot) \|_{L^q}}  dt\right)
\end{equation}
We will estimate the two integrals separately. For the first one we calculate
\begin{equation}
 \| \partial_t \tilde f_t(\cdot) \|_{L^q} = \| g-f \|_{L^q}\leq \|  g \|_{L^q}+\|  f \|_{L^q}\leq  2 q
\end{equation}
and 
\beq
\begin{split}
 \|  \tilde f_t(\cdot) \|^q_{L^q} &= \int \left( (1-t)f(\theta)+tg(\theta)\right)^q d\theta \\&\geq  \int  (1-t)^q f(\theta)^q d\theta + \int t^q g(\theta)^q  d\theta =
 (1-t)^q \|  f \|^q_{L^q}+ t^q \|  g \|^q_{L^q} {\ge \frac{2}{2^q} q^q}\;.
\end{split}
\eeq
Thus the first term can be estimated by {$4$}. In these estimates we made repeatedly use of the fact that all involved functions are positive.
For the second term we calculate using the H\"older inequality
\beq\begin{split}
\partial_t \| \tilde f_t(\cdot) \|_{L^q}
	&= \partial_t \left( \int \tilde f_t(\theta)^q d\theta\right)^{1/q} 
		= \left( \int \tilde f_t(\theta)^q d\theta\right)^{1/q-1} \int \tilde f_t(\theta)^{q-1} (g-f)\, d\theta\\
	&\leq \left( \int \tilde f_t(\theta)^q d\theta\right)^{1/q-1} \left( \int \tilde f_t(\theta)^{q}d\theta\right)^{(q-1)/q}  \left(\int (g-f)^q\, d\theta\right)^{1/q}\\
	&=\left(\int (g-f)^q\, d\theta\right)^{1/q} 
		=  \| g- f\|_{L^q} \leq {2q}\,, 
\end{split}\eeq
and thus the second term is bounded as well {by $4$}. This, in turn, proves the desired bound for the diameter of $\mathcal{U}_p$.
\end{proof}

Note that the upper bound above implies directly the boundedness of the diameter of $\Diff(S^1)/S^1$ with respect to the homogeneous $\dot W^{1,q}$ metric. 
We now use the lower bound, together with the Sobolev embedding theorem of Lemma~\ref{lem:critical_embedding}
to show the unboundedness of $\diam_{s,p} \Diff(S^1)$ in the critical case:
\begin{lemma}
Let $p\in(1,\infty)$. Then the diameter $\diam_{1+1/p,p} \Diff(S^1)$ of $\Diff(S^1)$ with respect to the right-invariant $W^{1+1/p,p}$-norm is unbounded. 
\end{lemma}
\begin{proof}
In Lemma~\ref{lem:Wsp_sphere} we have shown that $\diam_{1,q} \Diff(S^1)/S^1 > q$, and thus we also have $\diam_{1,q} \Diff(S^1)> q$.
In particular we have shown that for any $q\geq 1$ there is some $\vp^q\in \Diff(S^1)$ such that $\dist_{1,q}(\id,\vp^q)> {q}$.
Therefore, using Lemma~\ref{lem:critical_embedding},
we have 
\beq\begin{split}
{q} < \dist_{1,q}(\id,\vp^q) 
	&= \operatorname{inf} \int_0^1 \|\varphi^q_t\circ(\varphi_t^q)^{-1}\|_{W^{1,q}}\; dt\\
	&< \operatorname{inf} Cq^{1-\frac{1}{p}}\int_0^1  \|\varphi^q_t\circ(\varphi_t^q)^{-1}\|_{W^{1+{1/p},p}}\; dt\\
	&=Cq^{1-\frac{1}{p}}\dist_{1+{1/p},p}(\id,\vp^q).
\end{split}
\eeq
Thus $\dist_{1+{1/p},p}(\id,\vp^q) \ge Cq^{1/p}$
and taking $q\to \infty$ completes the proof. 
\end{proof}

\begin{remark}
It might be possible to use a similar argument to prove the second part of Question~\ref{qn:continuity_diameter}, i.e., that $\lim_{s\to (1 + 1/p)-} \diam_{s,p} \Diff(S^1) = \infty$, by controlling the embedding constants of $W^{s,p}$ into $W^{1,q}$ for $s\nearrow 1+1/p$ and an appropriate $q(s)\to \infty$.
\end{remark}

\subsection{The subcritical case $s<1+1/p$.}\label{sec:diams1_subcritical}
It remains to show the boundedness of the diameter for $s<1+1/p$. Towards this aim we will first show that if (controlled) arbitrary change in length (volume) has a bounded cost, then the diameter is finite:
\begin{lemma}
\label{lem:contractions_S_1}
Identify $S^1$ with the unit interval, and let $(s-1)p<1$.
For any $\lambda\in \mathbb{N}$ and $\delta\in (0,1)$, denote by $\psi_{\lambda,\delta}\in \Diff(S^1)$ a map satisfying $\psi_{\lambda,\delta}(x) = \lambda x$ for $x\in \Brk{0, \frac{1-\delta}{\lambda}}$.
If there exists $C=C(s,p)>0$, independent of $\lambda$ and $\delta$, such that 
\[
\dist_{W^{s,p}([0,1])}(\id, \psi_{\lambda,\delta}) <C \quad \text{for every $\lambda \in \mathbb{N}$ and $\delta\in (0,1)$},
\]
then
\[
\diam_{s,p}\Diff(S^1) \le \diam_{s,p}\Diff([0,1]) + 1 < 4C+1.
\]
\end{lemma}

\begin{proof}
Let $\vp\in \Diff(S^1)$; by translating, we can assume that $\vp(0)=0$.
{This translation} costs at most 1, hence the "$+1$" in the statement of the theorem.
We can always write $\vp = \vp_1 \circ \vp_2$, where $\supp \vp_1\subset [0,1-\delta]$, and $\supp \vp_2 \subset [\delta, 1]$ for some $\delta>0$.
Since $\dist(\id, \vp) \le \dist(\id, \vp_1) + \dist(\id, \vp_2)$, it is enough to prove that both $\dist(\id,\vp_1)$ and $\dist(\id, \vp_2)$ are smaller than  $2C$.
{Note that for} $\vp\in \Diff(S^1)$ with $\vp(0){=0}$ we have
\[
\dist_{W^{s,p}(S^1)}(\id,\vp) \le \dist_{W^{s,p}([0,1])}(\id,\vp),
\]
{hence} it is enough to prove the statement for $W^{s,p}([0,1])$.
Henceforth in this proof, we will only refer to $[0,1]$.

{Following the decomposition above}, from now on we will assume that $\supp \vp \subset [0,1-\delta]$ for some $\delta>0$, and consider $\vp$ as a diffeomorphisms of $\R$.
Denote 
\[
\vp^\lambda(x) = \frac{1}{\lambda}\vp(\lambda x) = \psi_{\lambda,\delta}^{-1} \circ \vp \circ \psi_{\lambda,\delta}(x),
\]
where the last equality holds because $\supp \vp \subset [0,1-\delta]$.
Using our assumption, we have
\[
\dist_{s,p}(\id ,\vp) = \dist_{s,p}(\id, \psi_{\lambda,\delta} \circ \vp^\lambda \circ \psi_{\lambda,\delta}^{-1}) < 2C + \dist_{s,p}(\id, \vp^\lambda).
\]
{We now show that $\dist_{s,p}(\id, \vp^\lambda)$ can be controlled by $\dist_{s,p}(\id ,\vp)$ times a small constant.} 
A direct calculation shows that the map $\vp(t,x)\mapsto \vp^\lambda(t,x)=\lambda^{-1}\vp(t,\lambda x)$ is a bijection between paths supported on $[0,1]$ to paths supported on $[0,1/\lambda]$, with the corresponding vector fields
\[
u_t^\lambda(x) = \frac{1}{\lambda}u_t(\lambda x).
\]
Note that
\[
\|u_t^\lambda\|_{\dot W^{s,p}} = \lambda^{(s-1)-\frac{1}{p}} \|u_t\|_{\dot W^{s,p}},\qquad 
\|u_t^\lambda\|_{\dot W^{1,p}} = \lambda^{-\frac{1}{p}} \|u_t\|_{\dot W^{1,p}},\qquad 
\|u_t^\lambda\|_{L^p} = \lambda^{-1-\frac{1}{p}} \|u_t\|_{L^p}
\]
where $\dot W^{s,p}$ refers to the $(s-1)$-Gagliardo seminorm on the derivative (if $s>1$), c.f., Lemma~\ref{lem:scaling}.
We therefore have
\[
\|u_t^\lambda\|_{W^{s,p}} \le \lambda^{(s-1)-\frac{1}{p}} \|u_t\|_{W^{s,p}},
\]
and hence
\[
\length_{W^{s,p}}(\vp^\lambda_t) \le \lambda^{(s-1)-\frac{1}{p}} \length_{W^{s,p}}(\vp_t).
\]
Therefore, {taking the infimum over all possible paths between $\id$ and $\vp$,}
we have
\[
\dist_{s,p}(\id, \vp^\lambda) \le \lambda^{(s-1)-\frac{1}{p}} \dist_{s,p}(\id,\vp).
\] 
We conclude that
\[
\dist_{s,p}(\id,\vp) < \frac{2C}{1-\lambda^{(s-1)-\frac{1}{p}}}.
\]
Since $(s-1)p<1$, taking $\lambda\to \infty$ concludes the proof.
\end{proof}

The following lemma shows that $\dist_{W^{s,p}([0,1])}(\id, \psi_{\lambda,\delta})$ is indeed uniformly bounded, by showing that the affine homotopy is uniformly bounded (in $\lambda$ and $\delta$) in the subcritical regime.
As the proof is a rather technical calculation, we postpone it to Appendix~\ref{app:technical_lemma}.

\begin{lemma}\label{lem:psikS1}
Let $s<1+1/p$. The there exists a sequence of maps $\psi_{\lambda,\delta}$ with $\psi_{\lambda,\delta}(x) = \lambda x$ for $x\in \Brk{0, \frac{1-\delta}{\lambda}}$ 
such that 
\[
\dist_{W^{s,p}([0,1])}(\id, \psi_{\lambda,\delta}) <C \quad \text{for every $\lambda\in \mathbb{N}$ and $\delta\in (0,1)$}
\]
where $C=C(s,p)$ is independent of $\lambda$. 
\end{lemma}

\subsection{The displacement energy}\label{sec:discplacementenergy_diffs1}
Finally, we discuss boundedness properties of the displacement energy as introduced in Definition~\ref{def:displacement_energy}. 
Even though we do not have sufficiently strong result relating the boundedness of the diameter and the displacement energy (see Appendix~\ref{sec:displacement_energy}),
we will now show that indeed in our case boundedness (unboundedness resp.) of the displacement energy of arbitrarily large open subsets of $S^1$ is closely related to the boundedness of the diameter.
While it is obvious that bounded diameter implies bounded displacement energy, we give below a direct, simpler proof for the boundedness of the displacement energy in the subcritical case.

\begin{proposition}\label{bounded_displacement_energy}
Identify $S^1$ with the interval $[0,1]$.
We then have the following bounds
\begin{enumerate}
\item For an interval $I$ with length smaller then $1/2$, we have
	\[
	E_{s,p}(I) \le 1/2 \quad \text{for all $s,p$.}
	\]
\item For every $s< 1+1/p$, there exists $C=C(s,p)>0$ such that
	\[
	E_{s,p}(I) < C \quad \text{for every open interval $I\subset[0,1]$}.
	\]
\item If $s > 1+1/p$ then there exists $c=c(s,p)>0$ such that
	\[
	E_{s,p}((0,1-\delta)) > c|\log \delta|.
	\]
\end{enumerate}
\end{proposition}

\begin{remark}
Note that we do not know whether the displacement energy is bounded or not in the critical case $s=1+1/p$.
\end{remark}

\begin{proof}
The first assertion follows by flowing for time $1$ along the constant vector field $u(t,x) = 1/2$, whose $W^{s,p}$-norm is $\frac{1}{2}$.

The last assertion follows from \eqref{eq:lower_bound_derivative}.
Indeed, if $\vp\in \Diff(S^1)$ such that $\vp((0,1-\delta))\cap (0,1-\delta) = \emptyset$, then $\vp((1-\delta, 1)) \supset (0,1-\delta)$, and therefore
\[
\delta \max(\vp') \ge \int_{1-\delta}^1 \vp'(x)\,dx = \vp(1) - \vp(1-\delta) \ge (1-\delta) - 0 = 1-\delta.
\]
Therefore $\max(\vp') > \frac{1-\delta}{\delta}$, and so by \eqref{eq:lower_bound_derivative}, $\dist_{s,p} (\id,\vp) \ge c|\log\delta|$.

The second assertion follows directly from Theorem~\ref{thm:diam_S_1}. In the following we will sketch an alternative simpler proof.
{C}onsider the vector field
\beq
\label{eq:u_alpha_eps}
u_{\alpha,\e}(x) = 
\begin{cases}
\e^{-\alpha}x & x\in [0,\e) \\
x^{1-\alpha} & x\in [\e,3/4) \\
4\brk{\frac{3}{4}}^{1-\alpha}(1-x) & x\in [3/4,1),
\end{cases}
\eeq
where $\alpha < 1+ \frac{1}{p} - s$ and $\e\ll 1$ to be determined (note that this vector field is simply $x^{1-\alpha}$ with linear interpolations to $0$ at $0\sim 1$).
A direct calculation shows that $\|u_{\alpha,\e}\|_{s,p} < C(s,p,\alpha)<\infty$ when $(\alpha + (s-1))p<1$.
Furthermore, we have that {the flow $\vp_t$ along $u_{\alpha,\e}$ (that is, }the solution to $\pl_t\vp_t = u_{\alpha,\e}\circ \vp_t$, $\vp_{0}(x)=x$) satisfies, for $x\ge \e$, and as long as $\vp_t(x)<3/4$,
\[
\vp_t(x) = (x^\alpha + \alpha t)^{1/\alpha} > (\alpha t)^{1/\alpha},
\]
hence in particular, for $t_0 = \frac{1}{\alpha 2^\alpha}$, $\vp^{\alpha,\e} = \vp_{t_0}$ satisfies
\[
\vp^{\alpha,\e}(0) = 0, \quad \vp^{\alpha,\e}(1) = 1, \quad \vp^{\alpha,\e}(x) > 1/2 \,\,\text{ for any $x > \e$}
\]
and
\[
\dist_{s,p}(\id, \vp^{\alpha,\e}) < \frac{1}{\alpha 2^\alpha}C(s,p,\alpha).
\]

Consider now the interval $I=(\delta ,1)$ for some $\delta <1/2$, and let $\e < \delta$.
Then 
\[
\psi = (\vp^{\alpha,\e})^{-1} \circ T_{1/2}\circ  \vp^{\alpha,\e}
\]
where $T_{1/2}$ is the translation by $1/2$, satisfies
\[
\psi(I) \subset (\vp^{\alpha,\e})^{-1} \circ T_{1/2} ( (1/2,1)) = (\vp^{\alpha,\e})^{-1} (0,1/2) \subset (0,\delta),
\]
hence $\psi(I)\cap I = \emptyset$.
Since
\[
\dist_{s,p}(\id,\psi) \le 2\dist_{s,p}(\id, \vp^{\alpha,\e}) + \dist_{s,p}(\id, T_{1/2}) \le  \frac{2}{\alpha 2^\alpha}C(s,p,\alpha) + \frac{1}{2}
\]
and the right-hand side is uniformly bounded in $\delta$, the proof is complete.
\end{proof}

\section{The diameter of  $\Diff(M)$ for compact manifolds in higher dimensions}\label{sec:diffM_diameter}
In this section we prove Theorem~\ref{thm:diam_compactM} in its full generality.
The main analytic ideas of the proof are similar to the one-dimensional case, however their adaptation to the higher dimensional settings is not always immediate;
in particular, for the boundedness proof, we need the diffeomorphism group to have a localization property, which we call {\em uniform fragmentation property}, described in Section~\ref{sec:diamM_subcritical}.

The structure of the section is as follows: in Section~\ref{sec:diamM_supercritical} we give a simple proof for the unboundedness of the diameter of $\Diff(\M)$ (for any manifold) when $s>1 + \frac{\dim\M}{p}$, and in Section~\ref{sec:diamM_critical} we give a more elaborate proof for the unboundedness in the case $s\ge 1 + \frac{\dim\M}{p}$.
We then present and discuss the uniform fragmentation property in Section~\ref{sec:diamM_subcritical}, and prove boundedness of $\Diff(\M)$ for manifolds that satisfy this property (like spheres), when $s< 1 + \frac{1}{p}$.
This completes the proof of Theorem~\ref{thm:diam_compactM}, as the zero diameter result of item~1 follows directly from  the results on vanishing geodesic distance in~\cite{jerrard2018vanishing,jerrard2019geodesic,bauer2013geodesic}.

As noted in the introduction, when $\dim \M > 1$, we have a gap in the range $s\in \left[1+ \frac{1}{p},1+ \frac{\dim \M}{p}\right)$.
We believe that in these cases the diameter is finite (assuming that the fragmentation property is satisfied); we show an indication for this in Section~\ref{sec:disp_energy_M}.

\subsection{The supercriticial case $s> 1+\frac{\operatorname{dim} \M}{p}$}\label{sec:diamM_supercritical}

\begin{lemma}
Let  $s>1+\frac{\operatorname{dim} \M}{p}$. Then the geodesic distance of the right-invariant $W^{s,p}$-norm on $\Diff(\M)$ satisfies
\beq
\label{eq:lower_bound_jacobian_dist}
\log |D\vp| \leq C \dist_{s,p} (\id,\vp)\;,
\eeq
where $C= C(s,p,\dim\M,\g)$, and $|D\vp|$ is the Jacobian determinant $\vp$ with respect to a chosen Riemannian metric $\g$ on $\M$.
It follows that $\diam_{s,p} \Diff(\M)$ is unbounded.
\end{lemma}

\begin{proof}
In this regime we have the Sobolev embedding $W^{s,p}(\M) \subset C^1(\M)$.
Let $\vp_t$ be a curve starting at $\id$ and ending at $\vp$, and let $u_t$ its associated vector field, that is $\pl_t\vp_t = u_t\circ \vp_t$.
Denote $\psi_t = |D\vp_t|$, the Jacobian determinant with respect to $\g$.
We have $\pl_t \psi_t = \div(u_t)\circ \vp \cdot \psi_t$, or in other words, $\pl_t (\log \psi_t)(x) = \div(u_t)(\vp_t(x))$.
Integrating this, using the fact that $\log \psi_0 = 0$, we have for any $(s-1)p>\operatorname{dim} \M$ and any $x\in \M$,
\[
\log |D\vp|(x) = \int_0^1 \pl_t (\log \psi_t(x)) \,dt \le \int_0^1 \|\div(u_t)\|_{L^\infty} \,dt \le C \int_0^1 \|u_t\|_{W^{s,p}} \,dt,
\]
where in the last inequality we used the above-mentioned Sobolev embedding.
By choosing $\vp$ with $|D\vp|$ arbitrarily large at a point, we get an arbitrarily large lower bound to the diameter.
\end{proof}

\begin{remark}
As in the one dimensional case, for $p=1$ this proof also works for the critical case $s = 1 +\dim \M$, as $W^{1+\dim \M,1}(\M) \subset C^{1}(\M)$.
\end{remark}

\subsection{The critical case $s= 1+\frac{\operatorname{dim} \M}{p}$}\label{sec:diamM_critical}
\begin{lemma}
Let $p\geq 1$ and $s= 1+\frac{\operatorname{dim} \M}{p}$. Then the diameter $\diam_{s,p} \Diff(\M)$ of $\Diff(\M)$ with respect to the right-invariant $W^{s,p}$-norm is unbounded. 
\end{lemma}

\begin{proof}
The proof of this result is inspired by connections between a homogeneous, degenerate $H^1$-metric --- called the information metric --- on the group of diffeomorphisms and the Fisher--Rao metric on the space of probability densities, 
see \cite{khesin2013geometry}.  Similarly as in the proof of the one-dimensional situation we aim to generalize this result to $W^{1,q}$-metrics for general $q\geq 1$.

We start by introducing a right-invariant, degenerate Sobolev (Finsler) metric of order one on the diffeomorphism group:
\begin{equation}
F_{\operatorname{Id}}(X)= \brk{\int_\M |\operatorname{div}(X)|^q \mu\;}^{{1/q}}
\end{equation}
where $\mu$ is some fixed volume form on $\M$ and $\operatorname{div}$ is the divergence with respect to $\mu$.
In the Riemannian case, $q=2$, this metric is also called  information metric due to its connections to the Fisher-Rao metric on the space of probability densities \cite{khesin2013geometry,bauer2015diffeomorphic}, which is the central object of interest in the area of information geometry \cite{amari2007methods}.

In particular, we have for $h\in T_\vp \Diff(\M)$
\beq
\label{eq:q_information_metric}
(F_{\vp(h)})^q = \int_\M |\operatorname{div}(h \circ \vp^{-1})|^q \mu = \int_\M |\operatorname{div}(h \circ \vp^{-1})\circ \vp|^q\, \vp^*\mu.
\eeq

Consider the mapping
\begin{equation}
\Phi: \operatorname{Diff}(\M) \to C^{\infty}(\M,\mathbb R), \qquad \varphi \mapsto q |D\varphi|^{1/q},
\end{equation}
where $|D\vp|$ the Jacobian determinant of $\vp$ with respect to $\mu$ (that is, $\vp^*\mu = |D\vp|\mu$).
Denote the function $\vp \mapsto |D\vp|$ by $\tilde \Phi$.
We have that
\beq
\label{eq:dF}
d\Phi(\vp).h = |D\vp|^{1/q-1} d\tilde \Phi(\vp).h = |D\vp|^{1/q} \div (h\circ \vp^{-1}) \circ \vp,
\eeq
since if $\vp(t)$ is a curve with $\vp(0) = \vp$ and $\dot \vp(0) = h$, then
\[
\begin{split}
\left.\frac{d}{dt}\right|_{t=0} \vp(t)^*\mu 
	&= \left.\frac{d}{dt}\right|_{t=0} (\vp(t)\circ \vp^{-1}\circ \vp)^*\mu
	=  \left.\frac{d}{dt}\right|_{t=0} \vp^*(\vp(t)\circ \vp^{-1})^*\mu
	=  \vp^*\brk{\left.\frac{d}{dt}\right|_{t=0} (\vp(t)\circ \vp^{-1})^*\mu} \\
	&=  \vp^*\brk{\mathcal{L}_{h\circ \vp^{-1}} \mu} 
	= \vp^*\brk{\div(h\circ \vp^{-1}) \mu}
	= \div(h\circ \vp^{-1}) \circ \vp \, \vp^*\mu \\
	&= \div(h\circ \vp^{-1}) \circ \vp \, |D\vp| \mu.
\end{split}
\] 
After equipping $C^{\infty}(M,\mathbb R)$ with the flat $L^q$ metric 
\begin{equation}
| \delta f|_f = \left(\int_{M} |\delta f|^q \; \mu\right)^{1/q},
\end{equation}
equations \eqref{eq:q_information_metric} and \eqref{eq:dF} imply that the mapping 
$\Phi$ is a Riemannian immersion onto the 
positive $L^q$-sphere in $C^{\infty}(M,\mathbb R)$ with image
\[
\mathcal{U}_q := \BRK{ f\in C^{\infty}(M;\R) \,:\, f>0, \|f\|_{L^q} = q \,(\Vol_\mu(\M))^{1/q}}.
\]

The proof now continues as in the one-dimensional case (Section~\ref{sec:diams1_critical}):
The intrinsic diameter of $\mathcal{U}_q$ with respect to the $L^q$ metric is bounded below by a constant (depending on $\M$) times $q$ --- indeed, by choosing $f\equiv q$ and a function $g\in \mathcal{U}_q$ that is large on a small set and close to zero on the rest of $M$, we have that $\|f-g\|_{L^q}\ge Cq$.
Since the $W^{1,q}$-norm controls the degenerate metric $F$, and the map $\Phi$ is an immersion, we have that
\[
\diam_{1,q}\Diff(\M) \ge Cq,
\]
for some constant $C$ independent of $q$.
The infinite diameter with respect to the $W^{1+\frac{\dim\M}{p},p}$-norm follows in the same way as in Section~\ref{sec:diams1_critical}, using the embedding $W^{1+\frac{\dim\M}{p},p} \subset W^{1,q}$ as in Lemma~\ref{lem:critical_embedding} and taking $q\to \infty$.
\end{proof}

\subsection{The subcritical case $s<1+\frac{1}{p}$}\label{sec:diamM_subcritical}
We start by introducing the geometric property of $\Diff(\M)$ which we need to prove the boundedness:
\begin{definition}\label{def:uni_fragment}
Let $\M$ be a compact finite dimensional manifold. The diffeomorphism group $\Diff(\M)$ is said to satisfy the {\em uniform fragmentation property}, if there exists a constant $K>0$ and a finite cover of $\M$ by balls $B_{\epsilon_\alpha}(x_\alpha)$ on which normal coordinates are defined, such that any diffeomorphism $\vp \in \Diff(M)$ can we written as a product of $K$ diffeomorphisms $\vp_i$, where each $\vp_i$ is only supported in one ball $B_{\epsilon_\alpha}(x_\alpha)$.
\end{definition}

\begin{remark}
\begin{enumerate}
\item If we denote by $K(\vp)$ the minimal number of diffeomorphisms $\vp_i$ needed in such a decomposition of $\vp\in \Diff(\M)$, then $K(\vp)<\infty$ for any fixed cover of any compact manifold --- this is the content of the well-known "fragmentation lemma" (see, e.g., \cite[Lemma~2.1.8]{banyaga2013structure}).
	The quantity $K(\vp)$ is sometimes referred to as the \emph{fragmentation norm} of $\vp$ with respect to this cover (see, e.g., in the context of homeomorphisms, \cite[Section~2.3]{mann2018large}).
	We ask for this fragmentation norm to be uniformly bounded, independent of $\vp$.
\item If, instead of a fixed finite cover, we consider a cover by all open (topological) balls, then $K(\vp)$ is a conjugation-invariant norm, in the sense of \cite{burago2008conjugation} (see Example~1.14 there). 
	As such, it is known to be uniformly bounded for many manifolds, see, e.g., \cite[Theorem~VI]{fukui2019uniform} for a recent account on this.
	Unfortunately, we cannot allow for the balls to be arbitrary, as they are fixed a-priori in the definition of the norm, as discussed in Section~\ref{sec:Sobolev_manifolds} (see also the proof of Proposition~\ref{pn:localization} below).
\end{enumerate}
\end{remark}

Next we show that the $n$-dimensional sphere $S^n$ satisfies the uniform fragmentation property:
\begin{proposition}
Let $n\ge 1$ and let $A,B\subset S^n$ be open geodesic balls, $A\cup B = S^n$ and $\overline{A},\overline{B}\ne S^n$.
Then $\Diff(S^n)$ has an uniform fragmentation property with respect to the cover $\{A,B\}$.
\end{proposition}
\begin{proof}
Let $\vp\in \Diff(S^n)$.
The proof consists of two steps: 
\begin{enumerate} 
\item First we split $\vp = \vp_B\circ  \vp_A$, such that $S^n \setminus \supp \vp_A$ contains a ball $U_A$ satisfying $\overline{U}_A\subset B^c \subset A$, and similarly to $\vp_B$.
\item We then show that each of $\vp_A$ and $\vp_B$ can be written as a composition of at most $3$ diffeomorphisms, each supported either in $A$ or in $B$.
\end{enumerate}

\paragraph{Step I:}
Fix $x\in \vp^{-1}(A^c)$, then there exists an open ball $V_B$, containing $x$, such that $\overline{V}_B\subset \vp^{-1}(A^c)$, and therefore $B^c\cap \vp(\overline{V_B}) = \emptyset$. 
We can choose $V_B$ small enough such that $B^c \setminus \overline{V_B} \ne \emptyset$ as well.
Then there exists $\vp_A\in \Diff(S^n)$ such that $\vp_A|_{V_B} = \vp|_{V_B}$ and $B^c\setminus \supp \vp_A \ne \emptyset$, hence there exists a ball $U_A$ such that $\overline{U}_A \subset B^c\setminus \supp \vp_A$.
Setting $\vp_B = \vp \circ \vp_A^{-1}$, we have that $\vp_B|_{\vp(\overline{V_B})} = \id$, hence $\vp(V_B) \subset A^c\setminus \supp \vp_B$.
Choosing a ball $U_B$ such that $\overline{U_B} \subset \vp(V_B)$ completes this step.

\paragraph{Step II:}
We prove the result for $\vp_A$; the case of $\vp_B$ is analogous.
We decompose $\vp_A= \vp_3\circ \vp_2 \circ \vp_1$, where each $\vp_i$ is supported in either $A$ or $B$.

First, we construct $\vp_1$ to make all the points that end in $B$ start in $B$.
Consider the set $\vp_A^{-1}(\overline B)$. 
Since $\overline{U}_A\cap \vp_A^{-1}(\overline B) = \emptyset$, we have
\beq
\label{eq:uniform_frag_1}
A\cap \vp_A^{-1}(\overline B) \subset A\setminus \overline{U}_A.
\eeq
Since $A \setminus \overline{U}_A$ is diffeomorphic to $A\cap B$, there exists a diffeomorphism $\vp_1$ with $\vp_1 (A \setminus \overline{U}_A) = A\cap B$.
Moreover, we can choose $\vp_1$ such that it is supported on $A$ {(see Figure~\ref{fig:fragmentation})}.

\begin{figure}
\centering
\begin{subfigure}{.48\textwidth}
  \centering
  \includegraphics[width=1\linewidth]{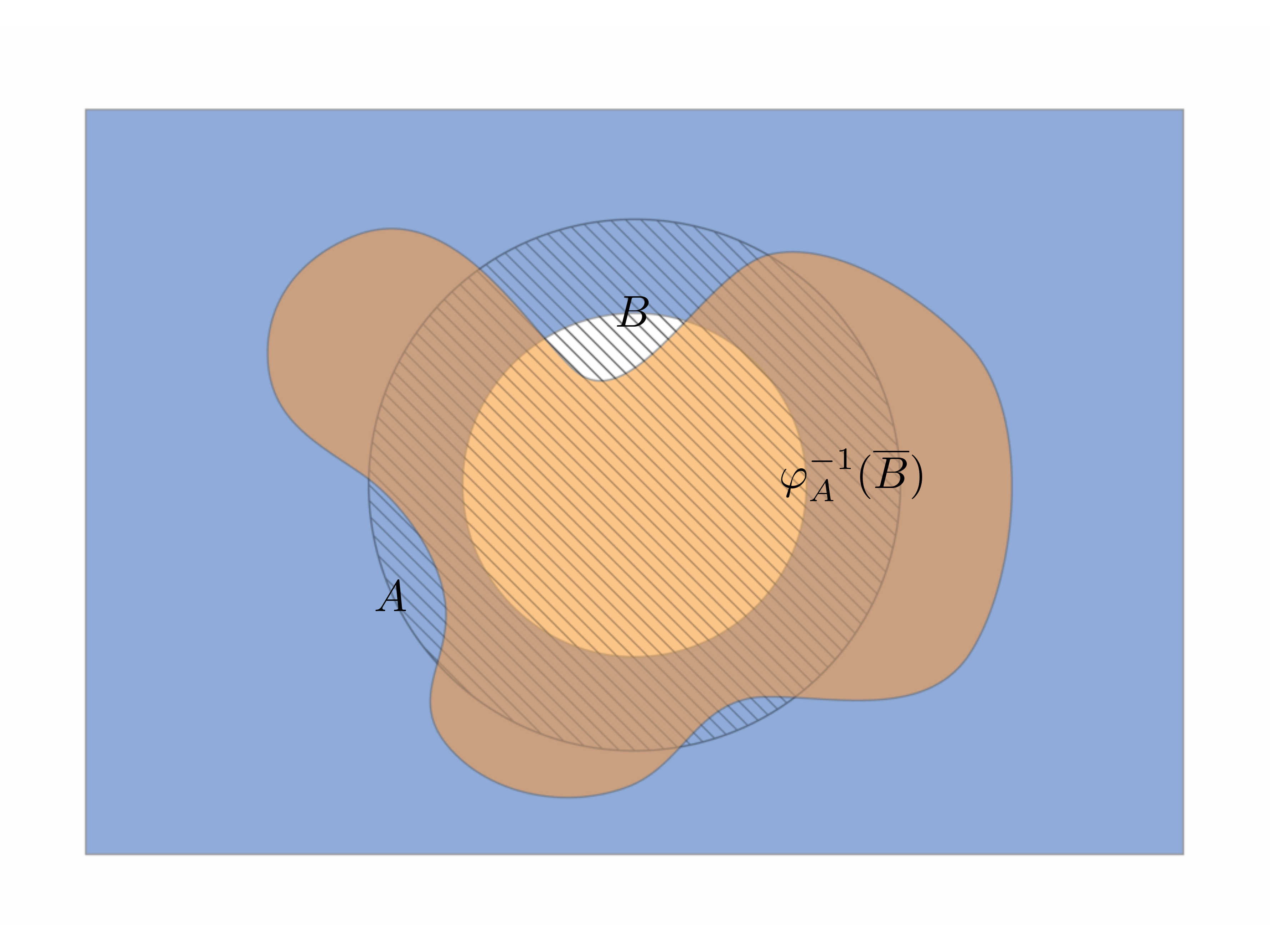}
\end{subfigure}
\begin{subfigure}{.48\textwidth}
  \centering
  \includegraphics[width=1\linewidth]{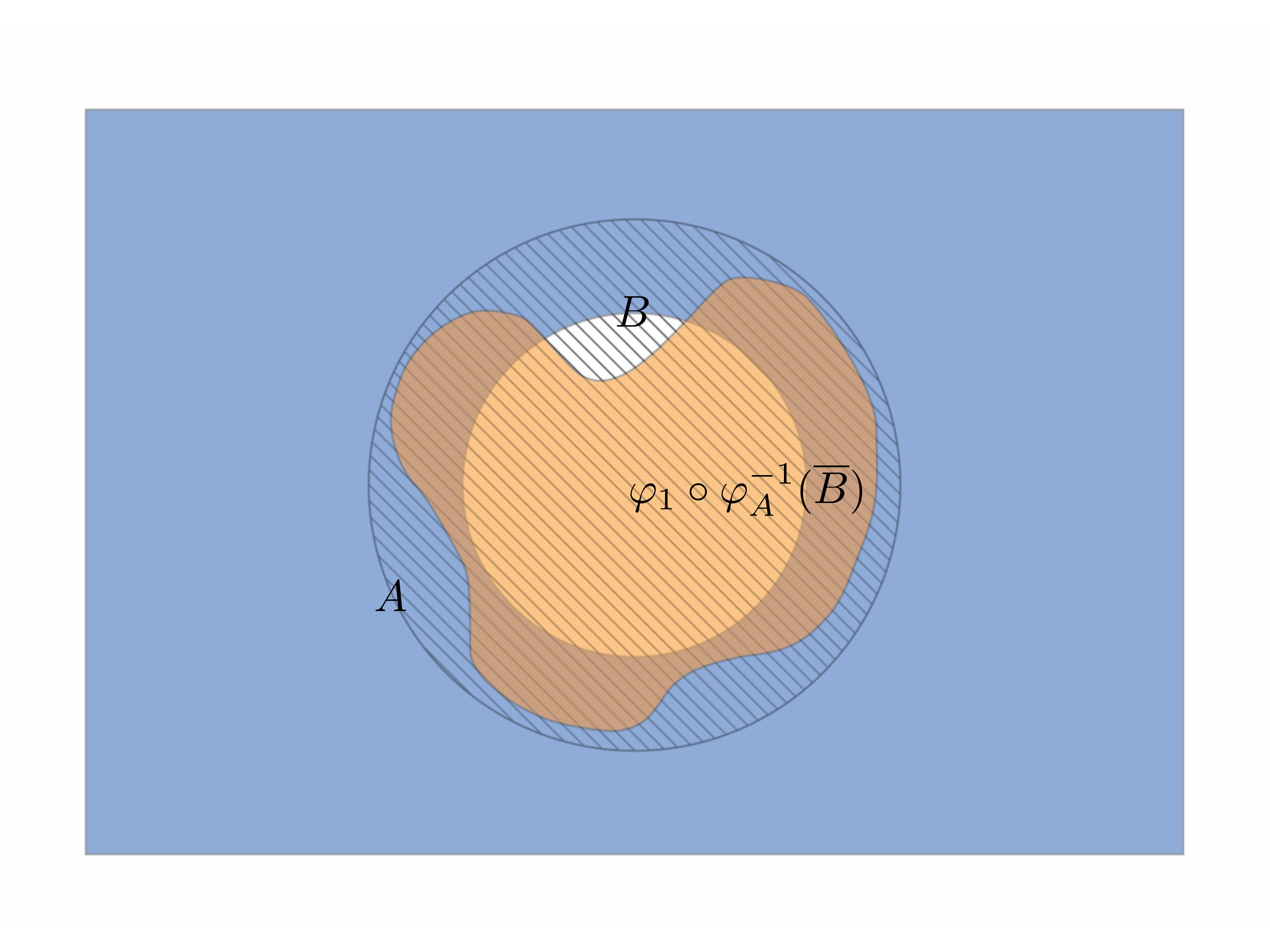}
\end{subfigure}
\caption{A sketch of the construction of $\vp_1$. The sets $A$ and $B$ that cover the sphere are the blue and crossed one, respectively. On the left-hand figure, the orange domain is $\vp_A^{-1}(\overline{B})$. Note that $\vp_A^{-1}(\overline{B})$ does not wrap around $A$ because of \eqref{eq:uniform_frag_1}.
The map $\vp_1$ is supported in $A$, and maps the set $\vp_A^{-1}(\overline{B})$ into $B$, as in the right-hand figure.}
\label{fig:fragmentation}
\end{figure}

We now have that
\beq
\label{eq:uniform_frag_2}
\vp_1 \circ \vp_A^{-1}(\overline B) \subset B.
\eeq
Indeed, by \eqref{eq:uniform_frag_1} we have that 
\[
\vp_1 \brk{A\cap \vp_A^{-1}(\overline B)} \subset A \cap B,
\]
and since $\vp_1$ is supported in $A$, we have also that 
\[
\vp_1 \brk{B\cap \vp_A^{-1}(\overline B)} \subset B.
\]

Next, we move all the points that end in a neighborhood of $A^c$ to their final destination.
Let $D$ be an open ball containing $A^c$, such that $\overline{D}\subset B$.
From \eqref{eq:uniform_frag_2} we have that $\vp_1\circ \vp_A^{-1} (\overline{D})$ is a closed ball contained in $B$.
Therefore, there exists a diffeomorphism $\psi$, supported in $B$, such that $\psi|_{\overline{D}} = \vp_1\circ \vp_A^{-1}|_{\overline{D}}$.
Define $\vp_2 := \psi^{-1}$.

Finally, define $\vp_3 := \vp_A \circ \vp_1^{-1} \circ \vp_2^{-1}$. 
We have that $\vp_3$ is supported in $A$: 
indeed, for any $x\in D$,
\[
\vp_3 (x) = \vp_A \circ \vp_1^{-1} \circ \psi(x) = \vp_A \circ \vp_1^{-1} \circ \vp_1 \circ \vp_A^{-1}(x) = x.
\]
\end{proof}

We will now continue with proving the boundedness of $\diam_{s,p} \Diff(\M)$ by showing that if $\Diff(\M)$ satisfies the uniform fragmentation property, then the question of finiteness of $\diam_{s,p} \Diff(\M)$ can be reduced to the finiteness of the diameter of diffeomorphisms groups of Euclidean balls.
\begin{proposition}\label{pn:localization}
Let $\M$ be a closed $n$-dimensional manifold, such that $\Diff(\M)$ satisfies the uniform fragmentation property with respect to some cover. 
Assume that $\diam_{s,p}\Diff_c(B) < \infty$, where $B$ is the unit ball of $\R^n$.
Then, $\diam_{s,p}\Diff(\M) < \infty$.
\end{proposition}

\begin{proof}
Note that by scaling, our assumption $\diam_{s,p}\Diff_c(B) < \infty$ implies the finite diameter of the compactly-supported diffeomorphism group of any Euclidean ball of arbitrary radius.

Let $\{B_{\e_{\alpha}}(x_\alpha)\}_{\alpha\in A}$ be an open cover of $\M$ by geodesic balls, with respect to which the uniform fragmentation property holds.
Consider now the cover $\{B_{\eta_\alpha}(x_\alpha)\}_{\alpha\in A}$, where $\eta_\alpha > \e_\alpha$ for each $\alpha \in A$, such that normal coordinates are defined on $B_{\eta_\alpha}(x_\alpha)$ as well.
To simplify notation, we denote $B^\alpha=B_{\e_\alpha}(x_\alpha)$ and $\tilde{B}^\alpha = B_{\eta_\alpha}(x_\alpha)$.
We will henceforth consider $W^{s,p}$-metrics on $\Diff(\M)$ with respect to $\{\tilde{B}^\alpha\}_{\alpha\in A}$ and a partition of unity $\rho_\alpha$ subordinate to this cover.

Let $\vp\in \Diff(\M)$.
By the uniform fragmentation property, there exists $\vp_1,\ldots,\vp_K\in \Diff(\M)$, with $\supp \vp_i \subset B^\alpha$ for some $\alpha\in A$, such that $\vp = \vp_K \circ \ldots \circ \vp_1$.
By right-invariance of the norm and the triangle inequality we have
\[
\dist_{s,p}(\id, \vp) \le \sum_{i=1}^K \dist_{s,p} (\id, \vp_i).
\]
Therefore, in order to prove that $\diam_{s,p}\Diff(\M)<\infty$, it is enough to prove that $\dist_{s,p} (\id, \vp_i)$ is uniformly bounded.
Therefore, we will henceforth assume that $\supp \vp \subset B^{\alpha_0}$ for some $\alpha_0\in A$.
Since, by assumption, $\diam_{s,p}\Diff_c(B) < \infty$, there exists a path $\vp^t$ (with vector fields $u^t$) from $\id$ to $\vp$, supported on $B^{\alpha_0}$, such that
\[
\length_{s,p}^{\R^n} (\vp^t) < C
\]
for some $C$ independent of $\vp$ (we identified $B^{\alpha_0}$ with a ball in $\R^n$ using normal coordinates as in the definition of the norm on $\mathcal{X}(\M)$).
This does not complete the proof as the support of $u^t$ intersects other coordinate balls involved in the definition of the $W^{s,p}$-norm on $\mathcal{X} (\M)$, and therefore these balls also contribute to the length of the path.
We now show that this contribution is uniformly bounded (and depends only on $\e_{\alpha_0}$ and $\eta_{\alpha_0}$).

Let $\rho_\alpha^{\alpha_0}$ be another partition of unity subordinate to $\tilde{B}^\alpha$, such that $\rho_{\alpha_0}^{\alpha_0}|_{B^{\alpha_0}} \equiv 1$.
With respect to this partition of unity, the length of the path $\vp^t$ is the same as in the coordinate chart (since, by definition of $\rho_\alpha^{\alpha_0}$, the support of $u^t$ intersect only the support of $\rho_{\alpha_0}^{\alpha_0}$), and is therefore bounded independent of $\vp$.

The norm with respect to $\rho_\alpha^{\alpha_0}$ is equivalent to the one with respect to the original $\rho_\alpha$, and therefore, since the choice of $\rho_\alpha^{\alpha_0}$ is independent of $\vp$ (depends only on $\e_{\alpha_0}$ and $\eta_{\alpha_0}$) and the cover is finite, we obtain a uniform bound on the length of $\vp^t$ with respect to our original norm as well.
\end{proof}

Using Proposition~\ref{pn:localization}, we now complete the proof of Theorem~\ref{thm:diam_compactM} by showing that $\diam_{s,p}\Diff_c(B) < \infty$ for $s<1+\frac{1}{p}$.

\begin{lemma}\label{lem:boundedness_ball}
Let $s<1+\frac1p$ and let $B$ denote the unit ball in $\mathbb R^n$. 
Then
\[
\diam_{s,p}\Diff_c(B)  < C(s,p,n) < \infty.
\]
\end{lemma}

\begin{proof}
We will use polar coordinate on $B_\epsilon$, in the standard notation
\begin{equation}
B=\left\{(r,\theta): r\in[0,1], \theta \in S^{n-1}  \right\}\;.
\end{equation}
Similar as in the one-dimensional situation the proof of this theorem will follow in two steps. 

\paragraph{Step I:} First we show that the boundedness of the diameter follows if (controlled) arbitrary change in volume has a bounded cost. 
To this end we denote by  $\Psi_{\lambda,\delta}\in \Diff(B)$ a map satisfying 
\beq
\label{eq:Psi_k_delta}
\Psi_{\lambda,\delta}(r,\theta) = (\psi_{\lambda,\delta}(r),\theta),
\eeq
where
\begin{equation}\label{eq:Psi_k_delta_1}
\psi_{\lambda,\delta}(r) = \lambda r \quad\text{for}\quad r\in \Brk{0, \frac{1-\delta}{\lambda}},
\end{equation}
with $\lambda \in \mathbb{N}$ and $\delta\in (0,1)$.
In addition we assume that for some $C>0$, independent of $\lambda$ and $\delta$, we have $\dist_{s,p}(\id ,\Psi_{\lambda,\delta}) < C$.
Let now $\vp \in \Diff_c(B)$ be an arbitrary diffeomorphism. 
Since it has  compact support, there exists $\delta >0$ such that
$$\supp \vp\subset \left\{[0,1-\delta]\times S^{n-1}\right\}\subset B.$$

Denote 
\[
\vp^\lambda(r,\theta) = \frac{1}{\lambda}\vp(\lambda r,\theta) = \Psi_{\lambda,\delta}^{-1} \circ \vp \circ \Psi_{\lambda,\delta}(r,\theta),
\]
where the last equality holds because $\supp \vp \subset [0,1-\delta]\times S^{d-1}$ and since multiplication by a scalar in polar coordinates is given by
$\lambda(r,\theta)=(\lambda r,\theta)$.
Using the boundedness of $\dist_{s,p}(\id, \Psi_{\lambda,\delta})$ we have
\[
\dist_{s,p}(\id ,\vp) = \dist_{s,p}(\id, \Psi_{\lambda,\delta} \circ \vp^\lambda \circ \Psi_{\lambda,\delta}^{-1}) < 2C + \dist_{s,p}(\id, \vp^\lambda).
\]
By a direct calculation, similar as in  Lemma~\ref{lem:contractions_S_1}, we have that the map $\vp(t,x)\mapsto \vp^\lambda(t,x)=\lambda^{-1}\vp(t,\lambda x)$ is a bijection between paths supported on $[0,1]\times S^{n-1}$ to paths supported on $[0,1/\lambda]\times S^{n-1}$, with the corresponding vector fields
\[
u_t^\lambda(r,\theta) = \frac{1}{\lambda}u_t(\lambda r,\theta).
\]
From here the proof of the the above claims follows exactly as in Lemma~\ref{lem:contractions_S_1}, using Lemma~\ref{lem:scaling} to obtain that
\[
\dist_{s,p}(\id, \vp^\lambda) \le \lambda^{(s-1)-\frac{n}{p}} \dist_{s,p}(\id,\vp).
\] 
Hence, by taking $\lambda \to \infty$, we have $\dist_{s,p}(\id ,\vp) \le 2C$.
Note that it follows from Lemma~\ref{lem:scaling} that this part of the proof holds whenever $s< 1 +\frac{n}{p}$, not merely when $s< 1 + \frac{1}{p}$.

\paragraph{Step II:}
It remains to bound the distance from the identity to $\Psi_{\lambda,\delta}$ independently of $\lambda$ and $\delta$.
Here we will rely on our construction from the one-dimensional case. From Lemma~\ref{lem:psikS1}  we know that there exists a curve $\psi_{\lambda,\delta}^t$ in $\Diff([0,1])$ from $\id_{[0,1]}$ to $\psi_{\lambda,\delta}$ such that
\beq
\label{eq:one_d_bound}
\length_{s,p}(\psi_{\lambda,\delta}^t) < C(s,p), \quad \text{for every $\lambda\in \mathbb{N}$ and $\delta\in (0,1)$}.
\eeq
Let $u_t$ be the vector field associated with the curve $\psi_{\lambda,\delta}^t$, and define a curve $\Psi_{\lambda,\delta}^t$ by flowing from the identity map, along the vector field
\[
U_t(x) = u_t(|x|)\frac{x}{|x|}.
\]
Obviously, $\Psi_{\lambda,\delta}^1=\Psi_{\lambda,\delta}$ satisfies assumptions \eqref{eq:Psi_k_delta}--\eqref{eq:Psi_k_delta_1}.

Since $u_t \in W^{s,p}_0(0,1)$ we have, using Corollary~\ref{cor:radial_vectorfield_norm} and \eqref{eq:one_d_bound}, that
there exists $C=C(s,p, n)>0$, independent of $\lambda$ and $\delta$, such that 
\[
\dist_{W^{s,p}(B)}(\id, \Psi_{\lambda,\delta}) \le \length_{s,p}(\Psi_{\lambda,\delta}^t) < C, \quad \text{for every $\lambda\in \mathbb{N}$ and $\delta\in (0,1)$},
\]
which completes the proof.
\end{proof}

\subsection{The displacement energy}\label{sec:disp_energy_M}

Since step I of the proof of Lemma~\ref{lem:boundedness_ball} holds for any $s < 1 +\frac{n}{p}$, the only ingredient needed for proving that $\diam_{s,p}\Diff(\M) <\infty$ for $s< 1 +\frac{n}{p}$, is a better vector field $U_t$ in step II.
That is, we need a better way of flowing from $\id$ to $\Psi_{\lambda,\delta}$.
An indication that this should be possible is the following proposition, which deals with the uniform boundedness of the displacement energy of sets in $\Diff(S^n)$ (see Definition~\ref{def:displacement_energy}).
Although, as discussed earlier, we do not know that bounded displacement energy is equivalent to bounded diameter, all our current examples are consistent with such a claim.
Moreover, the proof shows that an arbitrary radial change of volume (which is what $\Psi_{\lambda,\delta}$ does) is possible at a bounded cost whenever $s < 1 +\frac{n}{p}$ (at least when $s\in \mathbb{N}$), although the change of volume in the proof is not as controlled as the one induced by $\Psi_{\lambda,\delta}$.

\begin{proposition}\label{pn:displacement_energy_spheres}
The following bounds on the displacement energy of subsets of $S^n$ hold:
\begin{enumerate}
\item If $s > 1+n/p$ then there exists $c=c(s,p)>0$ such that
	\[
	E_{s,p}(S^n\setminus B_\delta) > c|\log \delta|,
	\]
	where $B_\delta$ is a ball of radius $\delta$ in $S^n$.
\item For every integer $k< 1+n/p$, there exists $C=C(k,p,n)>0$ such that
	\[
	E_{k,p}(A) < C \quad \text{for every closed set $A\subsetneq S^n$}.
	\]
\end{enumerate}
\end{proposition}

\begin{remark}
We expect that the same line of proof below also yields the uniform boundedness for any non-integer $s < 1 + n/p$; we did not pursue the more involved estimates for non-integer values of $s$ as the main point of this section is to indicate why we conjecture that $1 + n/p$ is the critical exponent.
However, for some non-integer exponents (in particular when $1+n/p$ is an integer) we could simply use the Sobolev embedding theorem (see, e.g., \cite{behzadan2015multiplication}):
	\[
	\|\cdot \|_{W^{s,p}} \lesssim \|\cdot\|_{W^{r,q}}, \quad \frac{1}{q} - \frac{r}{n} = \frac{1}{p} - \frac{s}{n}, \quad r>s.
	\]
	For example, consider $H^s = W^{s,2}$ in two dimensions.
	Then we know that for $s>2$ the displacement energy is not bounded, while for $s<2$, we can choose $q = \frac{2}{3-s} < 2$, and then
	\[
	\|\cdot \|_{H^s} \lesssim \|\cdot\|_{W^{2,q}}.
	\]
	The uniform boundedness of the displacement energy for $W^{2,q}(S^2)$, $q<2$ therefore implies the boundedness for $H^s$, $s<2$.
\end{remark}

\begin{proof}
We start by proving the unboundedness for large $s$:
\paragraph{Unboundedness for $s > 1 + n/p$.}
Let $\delta>0$ be small enough.
Denote $A_\delta = S^n \setminus B_\delta$.
If $\vp(A_\delta) \cap A_\delta = \emptyset$, then $\vp(A_\delta) \subset B_\delta$, and therefore
\[
\int_{A_\delta} |D\vp| \, \dVol = \Vol(\vp(A_\delta)) \le \Vol(B_\delta).
\]
Since $\Vol(A_\delta)$ is of order one, and $\Vol(B_\delta) = O(\delta^n)$, it follows that there exists a point $x\in S^n$ such that $|D\vp(x)| = O(\delta^n)$.
The first part of the proposition now follows immediately from the estimate \ref{eq:lower_bound_jacobian_dist}.
 
\paragraph{Boundedness for $k < 1 + n/p$.}
For simplicity, we endow $S^n$ with a round metric with diameter $1$, and consider the cover of $S^n$ with two balls of radius $3/4$, one centered at the south pole and the other at the north pole.
Let $A\subsetneq S^n$ be a closed set.
Then, there exists a ball of radius $\e>0$, disjoint of $A$.
Denote it by $B_\e$.
Since an arbitrary rotation of $S^n$ has a bounded cost, we can assume without loss of generality that $B_\e$ is centered at the south pole.

We now construct $\vp$ such that $\vp(A)\subset B_\e$ and $\dist_{k,p}(\id,\vp)$ is bounded independently of $A$ and $\e$.
Fix $\alpha \in (0, \frac{n}{p} + 1 - k)$. 
Let $u_{\alpha,\e}\in C_c^\infty((0,3/4))$ be such that 
\[
u_{\alpha,\e}(x) = 
\begin{cases}
0 & x\in [0,\e/2) \\
x^{1-\alpha} & x\in [\e,2/3)
\end{cases}
\]
and such that, for some $C>0$ independent of $\e$,
\beq\label{eq:bounds_u}
\Abs{u_{\alpha,\e}^{(j)} (x)} \le C \e^{1-\alpha - j}\qquad \forall \,\,x\in \Brk{\frac{\e}{2},\e}, \,\, j = 0\ldots k.
\eeq

Define now a vector field $U_{\alpha,\e}$ on the Euclidean ball of radius $3/4$ by 
\[
U_{\alpha,\e} (x) = u_{\alpha,\e}(|x|) \frac{x}{|x|}.
\]
A straightforward calculation shows that
\[
D^{(j)}U_{\alpha,\e}(x) = \sum_{i=0}^j \frac{u_{\alpha,\e}^{(i)}(|x|)}{|x|^{j-i}} G_{i,j}\brk{\frac{x}{|x|}},
\]
where $G_{i,j}$ is a tensor-valued polynomial (independent of $u_{\alpha,\e}$).
We therefore have that
\beq\label{eq:bound_DU}
\Abs{D^{(j)}U_{\alpha,\e}(x)} \lesssim \sum_{i=0}^j \frac{\Abs{u_{\alpha,\e}^{(i)}(|x|)}}{|x|^{j-i}}.
\eeq
We now evaluate $\|D^{(k)} U_{\alpha,\e}\|_p$, and show that it is independent of $\e$.
By \eqref{eq:bound_DU} it is enough to show that for every $i\le k$,
\beq\label{eq:bound_DU_2}
\int_0^{3/4} \Abs{\frac{u_{\alpha,\e}^{(i)}(r)}{r^{k-i}}}^p \,r^{n-1}\,dr < C
\eeq
for some $C$ independent of $\e$.
Indeed
\[
\begin{split}
\int_0^{3/4} \Abs{\frac{u_{\alpha,\e}^{(i)}(r)}{r^{k-i}}}^p \,r^{n-1}\,dr 
	&= \int_{\e/2}^{\e} \Abs{\frac{u_{\alpha,\e}^{(i)}(r)}{r^{k-i}}}^p \,r^{n-1}\,dr 
		+ \int_{\e}^{2/3} \Abs{\frac{u_{\alpha,\e}^{(i)}(r)}{r^{k-i}}}^p \,r^{n-1}\,dr 
		+ \int_{2/3}^{3/4} \Abs{\frac{u_{\alpha,\e}^{(i)}(r)}{r^{k-i}}}^p \,r^{n-1}\,dr
\end{split}
\]
The third addend on the right-hand side can obviously be bounded independently of $\e$, and therefore we can ignore it.
The second addend can be evaluated explicitly, using the fact that $u_{\alpha,\e}(r) = r^{1-\alpha}$ in this region:
\[
\begin{split}
\int_{\e}^{2/3} \Abs{\frac{u_{\alpha,\e}^{(i)}(r)}{r^{k-i}}}^p \,r^{n-1}\,dr 
	&= C_{\alpha,i} \int_{\e}^{2/3} \Abs{\frac{r^{1-\alpha-i}}{r^{k-i}}}^p \,r^{n-1}\,dr 
		= C_{\alpha,i}  \int_{\e}^{2/3} r^{n-1+(1-\alpha - k)p}\,dr \\
	&< C_{\alpha,i}  \int_{0}^{1} r^{n-1+(1-\alpha - k)p}\,dr = C(\alpha,i,n,k,p),
\end{split}
\]
where in the last inequality we used the fact that $n + (1-\alpha - k)p >0$ since $\alpha < \frac{n}{p} + 1 - k$.
As for the first addend, we have, using \eqref{eq:bounds_u}, that
\[
\begin{split}
\int_{\e/2}^{\e} \Abs{\frac{u_{\alpha,\e}^{(i)}(r)}{r^{k-i}}}^p \,r^{n-1}\,dr 
	&\le C\e^{(1-\alpha - i)p}\int_{\e/2}^\e r^{n-1 - (k-i)p}\,dr \\
	&\le C' \e^{(1-\alpha - i)p} \e^{n-(k-i)p} = C' \e^{n+(1-\alpha - k)p},
\end{split}
\]
which is uniformly bounded in $\e$ since $n + (1-\alpha - k)p >0$.
In the transition to the second line we use the fact that the lower bound of the integral is $\e/2$ rather than $0$, and therefore we get boundedness even if $kp>n$ (for the case $i=0$).
This completes the proof of \eqref{eq:bound_DU_2}.

The proof for $j<k$ is similar, and therefore we obtain that there exists $C=C(k,p,n)$, independent of $\e$, such that
\beq\label{eq:bound_DU_3}
\|U_{\alpha,\e}\|_{W^{k,p}(\R^n)} < C.
\eeq

Let $\psi_t$ be a flow along $U_{\alpha,\e}$.
Similar to the one dimensional case, after time $t_0 = \frac{1}{\alpha 2^\alpha}$ we have that
\[
|\psi_{t_0}(x)| > 1/2 \quad \text{ whenever $|x|>\e$}.
\]

We now consider $\psi_{t_0}$ as a diffeomorphism on $S^n$, using the normal coordinate chart centered at the south pole.
$|\psi_{t_0}(x)| > 1/2$ for $|x|>\e$ implies then that $\psi_{t_0}$ maps the complement of $B_\e$ to the northern hemisphere.
Note that, by the same arguments as in Proposition~\ref{pn:localization}, the bound \eqref{eq:bound_DU_3} implies that
\[
\dist^{S^n}_{k,p}(\id,\psi_{t_0}) < C(k,n,p), \quad \text{ independent of $\e$.}
\]
Let $R$ be a rotation of $S^n$ that maps the south pole to the north pole, and consider
\[
\vp = \psi_{t_0}^{-1} \circ R \circ \psi_{t_0}.
\]
Since $\psi_{t_0}$ maps the complement of $B_\e$ to the northern hemisphere, it follows that $\vp$ maps the complement of $B_\e$ into $B_\e$, and therefore $\vp(A) \subset B_\e$.

As the bound on $\dist_{k,p}(\id,\psi_{t_0})$ implies that $\dist_{k,p}(\id,\vp)$ is bounded independent of $\e$, the proof is complete.
\end{proof}

\section{The diameter of $\Diff_{c}(\mathbb R^n)$}\label{sec:diam_R_n}
In the following we will consider the base manifold to be $n$-dimensional Euclidean space, i.e., $\M = \R^n$. 
In this case it will turn out, that the diameter of the the diffeomorphism group is either zero or unbounded (depending on the order $s$). 
We believe that the analogous results are also true for diffeomorphism groups on more general non-compact manifolds, but for simplicity, we will restrict ourselves here to the Euclidean case.

\begin{proof2}{Theorem~\ref{thm:diam_diff_R_n}}
The zero diameter result follows directly from the vanishing geodesic distance results of \cite{jerrard2018vanishing,jerrard2019geodesic,bauer2013geodesic}. 
It remains to show that the diameter is unbounded otherwise.

For $sp>n$, the proof of positive geodesic distance \cite{bauer2018vanishing,jerrard2018vanishing} uses the Sobolev embedding $W^{s,p}(\R^n) \subset L^\infty(\R^n)$. It shows that for any $\vp\in \Diffc(\R^n)$ and any $x\in \R^n$,
\beq
\label{eq:Sobolev_embedding_bound}
|\vp(x) - x| \le C \dist_{s,p} (\id ,\vp).
\eeq
Here $C=C(s,p,n)>0$ is a constant depending on $s,p$ and $n$.  
By choosing $\vp(0)$ to be arbitrarily far away from the origin this shows that $\diam_{s,p}\Diffc(\R^n) = \infty$ for $sp> n$.

For $s\geq1$ a scaling argument yields the result independently from $n$ and $p$.  
Fix $\lambda > 0$, and for $\vp\in \Diffc(\R^n)$, define $\vp^\lambda(x) := \lambda^{-1} \vp(\lambda x)$.
It is easy to see that $\supp \vp^\lambda = \lambda^{-1} \supp \vp$, hence $\vp^\lambda \in \Diffc(\R^n)$.
{Similar arguments as in the proof of Lemma~\ref{lem:contractions_S_1} show that given a path $\vp_t$ from $\id$ to $\vp$ with a vector field $u_t$, $\vp_t^\lambda$ is a path from $\id$ to $\vp^\lambda$ with a vector field
\[
u_t^\lambda(x) = \frac{1}{\lambda} u_t(\lambda x).
\]
}
It follows from Lemma~\ref{lem:scaling} that
\[
\| u_t^\lambda \|_{L^p}^p = \frac{1}{\lambda^{p+n}} \|u_t\|_{L^p}^p \qquad \| u_t^\lambda \|_{\dot W^{1,p}}^p = \frac{1}{\lambda^{n}} \|u_t\|_{\dot W^{1,p}}^p,
\]
and therefore, for $\lambda < 1$,
\[
 \| u_t^\lambda \|_{W^{1,p}} > \frac{1}{\lambda^{n/p}} \|u_t\|_{W^{1,p}}.
\]
Since  $\vp_t \mapsto \vp_t^{\lambda}$ is a bijection between the paths from $\id$ to $\vp$ to the paths from $\id$ to $\vp^\lambda$, we have
\[
\dist_{1,p}(\id,\vp^\lambda ) \ge \frac{1}{\lambda^{n/p}} \dist_{1,p} (\id, \vp).
\]
Taking $\lambda \to 0$, we obtain that $\diam_{1,p} \Diffc(\R^n)=\infty$ (since from \cite{bauer2013geodesic,jerrard2018vanishing} we already know that $\dist_{1,p} (\id, \vp)$ is not zero).
{Since the $W^{s,p}$ norm for $s>1$ controls the $W^{1,p}$ norm, we obtain that $\diam_{s,p} \Diffc(\R^n) = \infty$ for any $s\ge 1$.}
\end{proof2}

\appendix
\section{Proof of Lemma~\ref{lem:psikS1}}\label{app:technical_lemma}
We consider the family of piecewise-linear maps $\psi_{\lambda+1,\delta}$
\[
\psi_{\lambda+1,\delta } =
\begin{cases}
(\lambda+1)x 											& x\in\Brk{0,\frac{1-\delta}{\lambda+1}} \vspace{.1cm}\\
\frac{\delta (\lambda+1)}{\lambda+\delta}\brk{x- \frac{1-\delta}{\lambda+1}}+1-\delta		& x\in\Brk{\frac{1-\delta}{\lambda+1},1}
\end{cases} =
\begin{cases}
(\lambda+1)x 											& x\in\Brk{0,\frac{1-\delta}{\lambda+1}} \vspace{.1cm}\\
\frac{\delta(\lambda+1)x + (1-\delta)\lambda}{\lambda+\delta} 					& x\in\Brk{\frac{1-\delta}{\lambda+1},1}
\end{cases}.
\]
Since piecewise-linear maps are not elements of the group of diffeomorphisms $\Diff(S^1)$ we have to smoothen the maps around the break points $\frac{1-\delta}{\lambda+1}$ and $0\sim1$. 
However, since the $W^{s,p}$-metric can be extended to the space of Lipschitz-maps (for $s<1+1/p$) and since the smoothening can be done in such a way that the change in the distance to identity is arbitrarily small, we ignore this in the following. 

In the following we will bound the length of the linear homotopy $\vp_t(x)$ between $\id$ and $\psi_{\lambda+1,\delta}$ which albeit being straightforward turns out to be a somewhat tedious calculation.
We bound the length below with respect to the $\dot{W}^{s,p}$ norm, under the assumption that $s>1$.
Boundedness with respect to the lower order parts of $W^{s,p}$ norm, as well as for $W^{s,p}$ norm for $s\le 1$, is similar, but simpler.
We have
\[
\vp_t(x) = (1-t)x + t \psi_{\lambda+1,\delta } (x) =
\begin{cases}
(1+\lambda t)x 							& x\in\Brk{0,\frac{1-\delta}{\lambda +1}}  \vspace{.1cm}\\
x + t\frac{(1-\delta)\lambda }{\lambda +\delta}(1-x) 		& x\in\Brk{\frac{1-\delta}{\lambda +1},1}
\end{cases}.
\]
Its inverse is then given by
\[
\vp_t^{-1}(y) =
\begin{cases}
\frac{y}{1+\lambda t} 								& y\in\Brk{0,\frac{(1-\delta)(1+\lambda t)}{\lambda +1}}  \vspace{.1cm}\\
\frac{y-1}{1-t \frac{(1-\delta)\lambda }{\lambda +\delta}} + 1 		& y\in\Brk{\frac{(1-\delta)(1+\lambda t)}{\lambda +1},1}
\end{cases},
\]
and its time derivative is
\[
\pl_t\vp_t(x) = \psi_{\lambda +1,\delta } (x) - x=
\begin{cases}
\lambda x 							& x\in\Brk{0,\frac{1-\delta}{\lambda +1}}  \vspace{.1cm}\\
\frac{(1-\delta)\lambda }{\lambda +\delta}(1-x) 		& x\in\Brk{\frac{1-\delta}{\lambda +1},1}
\end{cases}.
\]
The vector field $u_t$ defined by $\pl_t\vp_t = u_t \circ \vp_t$ is therefore
\[
u_t(y) = \pl_t \vp_t (\vp_t^{-1}(y)) =
\begin{cases}
\frac{\lambda y}{1+\lambda t} 												& y\in\Brk{0,\frac{(1-\delta)(1+\lambda t)}{\lambda +1}}  \vspace{.1cm}\\
\frac{(1-\delta)\lambda }{\lambda +\delta}\frac{1-y}{1-t \frac{(1-\delta)\lambda }{\lambda +\delta}} 		& y\in\Brk{\frac{(1-\delta)(1+\lambda t)}{\lambda +1},1}
\end{cases}
=
\begin{cases}
\frac{y}{t+\frac{1}{\lambda }} 										& y\in\Brk{0,\frac{(1-\delta)(1+\lambda t)}{\lambda +1}}  \vspace{.1cm}\\
\frac{(1-\delta)(1-y)}{(1-t)(1-\delta) + \delta(1+\frac{1}{\lambda })}	 		& y\in\Brk{\frac{(1-\delta)(1+\lambda t)}{\lambda +1},1}
\end{cases},
\]
and therefore
\[
u_t'(y) =
\begin{cases}
\frac{1}{t+\frac{1}{\lambda }} 									& y\in\Brk{0,\frac{(1-\delta)(1+\lambda t)}{\lambda +1}}  \vspace{.1cm}\\
\frac{-(1-\delta)}{(1-t)(1-\delta) + \delta(1+\frac{1}{\lambda })}	 		& y\in\Brk{\frac{(1-\delta)(1+\lambda t)}{\lambda +1},1}
\end{cases}.
\]

We now evaluate the $\dot{W}^{1+\sigma,p}$-norm of $u_t$, for $\sigma p < 1$. 
That is, we evaluate the $(\sigma,p)$-Gagliardo seminorm of $u_t'$, whose $p$th power is
\[
\iint_{\R^2} \frac{|u_t'(x) - u_t'(y)|^p}{|x-y|^{1+\sigma p}}\,dx\,dy 
= 2\iint_{y>x} \frac{|u_t'(x) - u_t'(y)|^p}{|x-y|^{1+\sigma p}}\,dx\,dy 
= 2\int_{-\infty}^\infty\int_0^\infty \frac{|u_t'(x) - u_t'(x+s)|^p}{s^{1+\sigma p}}\,ds\,dx.
\]
We split this double integral into different regions:
\[
\begin{split}
\int_{-\infty}^\infty&\int_0^\infty \frac{|u_t'(x) - u_t'(x+s)|^p}{s^{1+\sigma p}}\,ds\,dx \\
	&= \int_{-\infty}^0 \int_{-x}^{-x+\frac{(1-\delta)(1+\lambda t)}{\lambda +1}} \frac{\brk{t+\frac{1}{\lambda }}^{-p}}{s^{1+\sigma p}}\,ds\,dx 
	+ \int_{-\infty}^0 \int_{-x+\frac{(1-\delta)(1+\lambda t)}{\lambda +1}}^{-x+1} \frac{\brk{\frac{1-\delta}{(1-t)(1-\delta) + \delta(1+\frac{1}{\lambda })}}^p}{s^{1+\sigma p}}\,ds\,dx \\
	& \quad+\int_0^{\frac{(1-\delta)(1+\lambda t)}{\lambda +1}} \int_{-x+\frac{(1-\delta)(1+\lambda t)}{\lambda +1}}^{-x+1} \frac{\brk{\frac{1}{t+\frac{1}{\lambda }}+\frac{1-\delta}{(1-t)(1-\delta) + \delta(1+\frac{1}{\lambda })}}^p}{s^{1+\sigma p}}\,ds\,dx \\
	&\quad+ \int_{\frac{(1-\delta)(1+\lambda t)}{\lambda +1}}^1 \int_{-x+1}^\infty \frac{\brk{\frac{1-\delta}{(1-t)(1-\delta) + \delta(1+\frac{1}{\lambda })}}^p}{s^{1+\sigma p}}\,ds\,dx .
\end{split}
\]

We now evaluate each of the four integrals in the right-hand side separately.
We will use repeatedly the following: for $\alpha \in (0,1)$ and $a>0$,
\[
\lim_{x\to \infty} (x+a)^\alpha - x^\alpha = 0,
\]
and
\[
(1-x)^\alpha \ge 1 - x^\alpha \qquad x\in[0,1].
\]
All the constants $C$ below are $C=C(p,\sigma)>0$, independent of $\lambda $, $\delta$ and $t$.

For the first integral we have:
\[
\begin{split}
\brk{t+\frac{1}{\lambda }}^{-p}\int_{-\infty}^0 \int_{-x}^{-x+\frac{(1-\delta)(1+\lambda t)}{\lambda +1}} \frac{1}{s^{1+\sigma p}}\,ds\,dx
	&=  \brk{t+\frac{1}{\lambda }}^{-p}\frac{1}{\sigma p} \int_{-\infty}^0 \brk{(-x)^{-\sigma p} - \brk{-x+\frac{(1-\delta)(1+\lambda t)}{\lambda +1}}^{-\sigma p}} \,dx \\
	&=  \brk{t+\frac{1}{\lambda }}^{-p}\frac{1}{\sigma p} \int_{0}^\infty \brk{x^{-\sigma p} - \brk{x+\frac{(1-\delta)(1+\lambda t)}{\lambda +1}}^{-\sigma p}} \,dx \\
	&=  \brk{t+\frac{1}{\lambda }}^{-p}\frac{1}{(1-\sigma p)\sigma p} \brk{x^{1-\sigma p} - \brk{x+\frac{(1-\delta)(1+\lambda t)}{\lambda +1}}^{1-\sigma p}}_0^\infty \\
	&=  \brk{t+\frac{1}{\lambda }}^{-p}\frac{1}{(1-\sigma p)\sigma p} \brk{\frac{(1-\delta)(1+\lambda t)}{\lambda +1}}^{1-\sigma p} \\
	&< C \brk{t+\frac{1}{\lambda }}^{-p+(1-\sigma p)} < C t^{-p+(1-\sigma p)}.
\end{split}
\]

The second integral can be bounded via: 
\[
\begin{split}
&\brk{\frac{1-\delta}{(1-t)(1-\delta) + \delta(1+\frac{1}{\lambda })}}^p\int_{-\infty}^0 \int_{-x+\frac{(1-\delta)(1+\lambda t)}{\lambda +1}}^{-x+1} \frac{1}{s^{1+\sigma p}}\,ds\,dx \\
	&\qquad\qquad = \frac{1}{\sigma p} \brk{\frac{1-\delta}{(1-t)(1-\delta) + \delta(1+\frac{1}{\lambda })}}^p \int_{-\infty}^0 \brk{\brk{\frac{(1-\delta)(1+\lambda t)}{\lambda +1}-x}^{-\sigma p} - (1-x)^{-\sigma p}} \,dx \\
	&\qquad\qquad = \frac{1}{\sigma p} \brk{\frac{1-\delta}{(1-t)(1-\delta) + \delta(1+\frac{1}{\lambda })}}^p \int^{\infty}_0 \brk{\brk{\frac{(1-\delta)(1+\lambda t)}{\lambda +1}+x}^{-\sigma p} - (1+x)^{-\sigma p}} \,dx \\
	&\qquad\qquad = \frac{1}{(1-\sigma p)\sigma p} \brk{\frac{1-\delta}{(1-t)(1-\delta) + \delta(1+\frac{1}{\lambda })}}^p \brk{\brk{\frac{(1-\delta)(1+\lambda t)}{\lambda +1}+x}^{1-\sigma p} - (1+x)^{1-\sigma p}}_0^\infty \\
	&\qquad\qquad = \frac{1}{(1-\sigma p)\sigma p} \brk{\frac{1-\delta}{(1-t)(1-\delta) + \delta(1+\frac{1}{\lambda })}}^p \brk{1 - \brk{\frac{(1-\delta)(1+\lambda t)}{\lambda +1}}^{1-\sigma p}} \\
	&\qquad\qquad \le \frac{1}{(1-\sigma p)\sigma p} \brk{\frac{1-\delta}{(1-t)(1-\delta) + \delta(1+\frac{1}{\lambda })}}^p \brk{1 - \frac{(1-\delta)(1+\lambda t)}{\lambda +1}}^{1-\sigma p}\\
	&\qquad\qquad = \frac{1}{(1-\sigma p)\sigma p} \brk{\frac{1-\delta}{(1-t)(1-\delta) + \delta(1+\frac{1}{\lambda })}}^p\brk{\delta + (1-\delta)\frac{\lambda }{\lambda +1}(1-t)}^{1-\sigma p}\\
	&\qquad\qquad < \frac{1}{(1-\sigma p)\sigma p}\brk{\frac{1-\delta}{(1-t)(1-\delta) + \delta}}^p \brk{\delta + (1-\delta)(1-t)}^{1-\sigma p}\\
	&\qquad\qquad = \frac{(1-\delta)^p}{(1-\sigma p)\sigma p} \brk{\delta + (1-\delta)(1-t)}^{-p+(1-\sigma p)}
		< C(1-t)^{-p+(1-\sigma p)}.
\end{split}
\]

Simirlarly we calcualte for the third integral:
\[
\begin{split}
&\brk{\frac{1}{t+\frac{1}{\lambda }}+\frac{1-\delta}{(1-t)(1-\delta) + \delta(1+\frac{1}{\lambda })}}^p \int_0^{\frac{(1-\delta)(1+\lambda t)}{\lambda +1}} \int_{-x+\frac{(1-\delta)(1+\lambda t)}{\lambda +1}}^{-x+1} \frac{1}{s^{1+\sigma p}}\,ds\,dx \\
	&\qquad\qquad = p\brk{\brk{t+\frac{1}{\lambda }}^{-p}+\brk{\frac{1-\delta}{(1-t)(1-\delta) + \delta(1+\frac{1}{\lambda })}}^p} \int_0^{\frac{(1-\delta)(1+\lambda t)}{\lambda +1}} \int_{-x+\frac{(1-\delta)(1+\lambda t)}{\lambda +1}}^{-x+1} \frac{1}{s^{1+\sigma p}}\,ds\,dx \\
	&\qquad\qquad = \frac{1}{\sigma}\brk{\brk{t+\frac{1}{\lambda }}^{-p}+\brk{\frac{1-\delta}{(1-t)(1-\delta) + \delta(1+\frac{1}{\lambda })}}^p} \int_0^{\frac{(1-\delta)(1+\lambda t)}{\lambda +1}} \brk{\brk{\frac{(1-\delta)(1+\lambda t)}{\lambda +1}-x}^{-\sigma p} - \brk{1-x}^{-\sigma p}} \,dx \\
	&\qquad\qquad = \frac{1}{(1-\sigma p)\sigma}\brk{\brk{t+\frac{1}{\lambda }}^{-p}+\brk{\frac{1-\delta}{(1-t)(1-\delta) + \delta(1+\frac{1}{\lambda })}}^p} 
		\brk{ \brk{1-x}^{1-\sigma p} - \brk{\frac{(1-\delta)(1+\lambda t)}{\lambda +1}-x}^{1-\sigma p}}_0^{\frac{(1-\delta)(1+\lambda t)}{\lambda +1}} \\
	&\qquad\qquad = \frac{1}{(1-\sigma p)\sigma}\brk{\brk{t+\frac{1}{\lambda }}^{-p}+\brk{\frac{1-\delta}{(1-t)(1-\delta) + \delta(1+\frac{1}{\lambda })}}^p} 
		\brk{\brk{1-\frac{(1-\delta)(1+\lambda t)}{\lambda +1}}^{1-\sigma p} - 1 +  \brk{\frac{(1-\delta)(1+\lambda t)}{\lambda +1}}^{1-\sigma p}} \\
	&\qquad\qquad \le \frac{1}{(1-\sigma p)\sigma}\Brk{\brk{t+\frac{1}{\lambda }}^{-p}\brk{\frac{(1-\delta)(1+\lambda t)}{\lambda +1}}^{1-\sigma p}
		+\brk{\frac{1-\delta}{(1-t)(1-\delta) + \delta(1+\frac{1}{\lambda })}}^p\brk{1-\frac{(1-\delta)(1+\lambda t)}{\lambda +1}}^{1-\sigma p}} \\
	&\qquad\qquad \le \frac{1}{(1-\sigma p)\sigma}\Brk{\brk{t+\frac{1}{\lambda }}^{-p}\brk{\frac{(1-\delta)(1+\lambda t)}{\lambda +1}}^{1-\sigma p}
		+\brk{\frac{1-\delta}{(1-t)(1-\delta) + \delta(1+\frac{1}{\lambda })}}^p\brk{\delta + (1-\delta)\frac{\lambda }{\lambda +1}(1-t)}^{1-\sigma p}}\\
	&\qquad\qquad < \frac{1}{(1-\sigma p)\sigma}\Brk{\brk{t+\frac{1}{\lambda }}^{-p}\brk{\frac{1+\lambda t}{\lambda +1}}^{1-\sigma p}
		+\brk{\frac{1}{(1-t)(1-\delta) + \delta}}^p\brk{\delta + (1-\delta)(1-t)}^{1-\sigma p}} \\
	&\qquad\qquad < \frac{1}{(1-\sigma p)\sigma}\Brk{\brk{t+\frac{1}{\lambda }}^{-p+(1-\sigma p)}
		+\brk{\delta + (1-\delta)(1-t)}^{-p+(1-\sigma p)}} \\
	&\qquad\qquad < C\brk{t^{-p+(1-\sigma p)}
		+(1-t)^{-p+(1-\sigma p)}} \\
\end{split}
\]

Finally the last integral can be bounded by:
\[
\begin{split}
&\brk{\frac{1-\delta}{(1-t)(1-\delta) + \delta(1+\frac{1}{\lambda })}}^p \int_{\frac{(1-\delta)(1+\lambda t)}{\lambda +1}}^1 \int_{-x+1}^\infty \frac{1}{s^{1+\sigma p}}\,ds\,dx \\
	&\qquad\qquad = \frac{1}{\sigma p} \brk{\frac{1-\delta}{(1-t)(1-\delta) + \delta(1+\frac{1}{\lambda })}}^p \int_{\frac{(1-\delta)(1+\lambda t)}{\lambda +1}}^1 (1-x)^{-\sigma p} \, dx \\
	&\qquad\qquad	=\frac{1}{(1-\sigma p)\sigma p} \brk{\frac{1-\delta}{(1-t)(1-\delta) + \delta(1+\frac{1}{\lambda })}}^p \left. (1-x)^{1-\sigma p}\right|_1^{\frac{(1-\delta)(1+\lambda t)}{\lambda +1}} \\
	&\qquad\qquad	=\frac{1}{(1-\sigma p)\sigma p} \brk{\frac{1-\delta}{(1-t)(1-\delta) + \delta(1+\frac{1}{\lambda })}}^p \brk{1-\frac{(1-\delta)(1+\lambda t)}{\lambda +1}}^{1-\sigma p} \\
	&\qquad\qquad	=\frac{1}{(1-\sigma p)\sigma p} \brk{\frac{1-\delta}{(1-t)(1-\delta) + \delta(1+\frac{1}{\lambda })}}^p \brk{\delta + (1-\delta)\frac{\lambda }{\lambda +1}(1-t)}^{1-\sigma p} \\
	&\qquad\qquad	<\frac{(1-\delta)^p}{(1-\sigma p)\sigma p} \brk{\frac{1}{(1-t)(1-\delta) + \delta}}^p \brk{\delta + (1-\delta)(1-t)}^{1-\sigma p} \\
	&\qquad\qquad	=\frac{(1-\delta)^p}{(1-\sigma p)\sigma p} \brk{\delta + (1-\delta)(1-t)}^{-p+(1-\sigma p)}
		< C(1-t)^{-p+(1-\sigma p)}.
\end{split}
\]

Overall we obtained
\[
\|u'_t\|_{\dot{W}^{\sigma,p}(\R)} < C\brk{(1-t)^{-p+(1-\sigma p)} + t^{-p+(1-\sigma p)}}^{1/p} 
	< C\brk{(1-t)^{-1+\frac{1-\sigma p}{p}}+ t^{-1+\frac{1-\sigma p}{p}}}
\]
where we used the fact that $(1+x)^\alpha < 1 + x^\alpha$ for $x>0$ and $\alpha\in (0,1)$.

We therefore have, using the fact that $1-\sigma p >0$, that
\[
\int_0^1 \|u'_t\|_{\dot{W}^{\sigma,p}(\R)} \, dt \le C \int_0^1 C\brk{(1-t)^{-1+\frac{1-\sigma p}{p}}+ t^{-1+\frac{1-\sigma p}{p}}} \,dt
	= 2C\frac{p}{1-\sigma p},
\]
which is a bound independent of $\lambda $ and $\delta$.

\section{Sobolev norms of radial functions}
In this section we prove a technical lemma on Sobolev functions, which is used in Section~\ref{sec:diamM_subcritical}.

\begin{lemma}\label{lem:radialfunctionnorm}
Let $n>1$, and define the operator $T:C_c^\infty((0,1))\to C_c^\infty(B_1(\R^n))$ by
\[
Tf(x) = f(|x|).
\]
Then for every $s\ge 0$ and $p\ge 1 $, we have
\[
\|Tf\|_{W^{s,p}} \le C \|f\|_{W^{s,p}},
\]
for some $C=C(s,p,n)>0$ independent of $f$.
That is, $T:W^{s,p}_0(0,1) \to W^{s,p}_0(B_1(\R^n))$ is a bounded operator for every $s\ge 0$ and $p\ge 1$.
\end{lemma}

\begin{proof}
\paragraph{Step I: integer Sobolev spaces}
We first prove the theorem for $W^{k,p}$ norms, where $k$ is an integer.
For $k=0$, moving to polar coordinates, we have
\[
\|Tf\|_{L^p}^p = \int_{B_1(\R^n)} |Tf(x)|^p \,dx = \w_n \int_0^1 |f(r)|^p r^{n-1}\,dr \le  \w_n\, \|f\|_{L^p}^p,
\]
where $\w_n$ is the measure of the $(n-1)$-dimensional unit sphere.
For $k=1$, we note that $D(Tf)(x) = f'(|x|)\frac{x}{|x|}$, hence $|D(Tf)(x)| = |f'(r)|$ and the estimate is similar.
Differentiating further, we have for $k=2$
\[
D^2(Tf)(x) = \brk{f''(|x|) - \frac{f'(|x|)}{|x|}}\frac{x}{|x|}\otimes \frac{x}{|x|} + \frac{f'(|x|)}{|x|} \id,
\]
and for higher derivatives we obtain
\[
D^k(Tf)(x) = \sum_{j=1}^k \frac{f^{(j)}(|x|)}{|x|^{k-j}} G_j^k\brk{\frac{x}{|x|}},
\]
where $G_j^k$ are smooth $k$-tensor-valued functions on {$S^{n-1}$}, which are independent of $f$.

In order to prove boundedness we need to prove that for $j\le k$ we have that
\[
\int_0^1 \Abs{\frac{f^{(j)}(r)}{r^{k-j}}}^p\,r^{n-1}\,dr \le \int_0^1 \Abs{f^{(k)}(r)}^p\, dr.
\]
This follows from Jensen's inequality: For $k=1$, we have
\[
\begin{split}
\int_0^1 \Abs{\frac{f'(r)}{r}}^p r^{n-1} \,dr & = \int_0^1 \Abs{\frac{1}{r} \int_0^r f''(t)\,dt}^p r^{n-1} \,dr 
	\le \int_0^1 \brk{\frac{1}{r} \int_0^r \Abs{f''(t)}^p \,dt} r^{n-1} \,dr \\
	& = \int_0^1 \int_0^r \Abs{f''(t)}^p \,dt \,r^{n-2} \,dr
	\le \int_0^1 \Abs{f''(t)}^p \,dt \cdot \int_0^1 \,r^{n-2} \,dr \\
	&= \frac{1}{n-1} \int_0^1 \Abs{f''(t)}^p \,dt.
\end{split}
\]
For $k=2$ we have
\[
\begin{split}
\int_0^1 \Abs{\frac{f'(r)}{r^2}}^p r^{n-1} \,dr 
	& = \int_0^1 \Abs{\frac{1}{r} \int_0^r \frac{1}{r} \int_0^t f^{(3)}(s)\,ds\, dt}^p r^{n-1} \,dr 
	 \le \int_0^1 \frac{1}{r} \int_0^r \Abs{ \frac{1}{r} \int_0^t f^{(3)}(s)\,ds}^p \, dt \, r^{n-1} \,dr \\
	 &= \int_0^1 \frac{1}{r} \int_0^r \frac{t^p}{r^p}\Abs{ \frac{1}{t} \int_0^t f^{(3)}(s)\,ds}^p \, dt \, r^{n-1} \,dr 
	 \le \int_0^1 \frac{1}{r} \int_0^r \frac{t^p}{r^p} \frac{1}{t} \int_0^t \Abs{f^{(3)}(s)}^p \,ds \, dt \, r^{n-1} \,dr \\
	 &\le \int_0^1 \Abs{f^{(3)}(s)}^p \,ds \cdot \int_0^1 \frac{1}{r} \int_0^r \frac{t^p}{r^p} \frac{1}{t} \, dt \, r^{n-1} \,dr 
	 = \frac{1}{p(n-1)} \int_0^1 \Abs{f^{(3)}(s)}^p \,ds.
\end{split}
\]
The result for higher values of $k$ follows in a similar manner.

\paragraph{Step II: interpolation}
Assume for now that $s\in (0,1)$.
Since $B_1(\R^n)$ is a convex set, we have that the $W^{s,p}(\R^n)$ norm on functions supported on $B_1(\R^n)$ {(the Gagliardo/Slobodeckij norm)} is equivalent to the norm of the real interpolation space 
$$\chi_0^{s,p}(B_1(\R^n)) = (L^p(B_1(\R^n)),W^{1,p}(B_1(\R^n)))_{s,p},$$ 
defined by
\[
\|f\|_{\chi_0^{s,p}(B_1(\R^n))}^p = \int_0^\infty \brk{\frac{K(t,f)}{t^s}}\,\frac{dt}{t} \qquad K(t,f) = \inf_{g\in C_0^\infty(B_1(\R^n))} \brk{\|f-g\|_{L^p(B_1(\R^n))} + t\|g\|_{W^{1,p}(B_1(\R^n))}}.
\]
See \cite[Theorem~4.7]{brasco2019note}.\footnote{In \cite{brasco2019note} the interpolation is defined with respect to the homogeneous $\dot{W}^{1,p}$ norm, but this does not matter as it is, by the Poincar\'e inequality, equivalent to the full $W^{1,p}$ norm on the space $W_0^{1,p}$ which we are considering. Similarly, the equivalence there is shown {between the interpolation space and} the homogeneous $\dot{W}^{s,p}$ norm, which is again equivalent to the full norm \cite[Section~2.3]{brasco2019note}.}
Since $\chi_0^{s,p}(B_1(\R^n))$ is an interpolation space, the map $T$ is bounded as a map $L^p([0,1])\to L^p(B_1(\R^n))$  and as a map $W_0^{1,p}([0,1])\to W_0^{1,p}(B_1(\R^n))$ and thus is also bounded as a map between the corresponding interpolation spaces $\chi_0^{s,p}([0,1])\to \chi_0^{s,p}(B_1(\R^n))$ (see, e.g., \cite[Section~2.3, Theorem~3]{salo08function}).

When $s=k+\sigma$, the proof is similar: $T$ is bounded as a map of between the interpolation spaces $(\dot{W}^{k,p}(0,1),\dot{W}^{k+1,p}(0,1))_{\sigma,p}\to(\dot{W}^{k,p}(B_1(\R^n)),\dot{W}^{k+1,p}(B_1(\R^n)))_{\sigma,p}$, since by the previous step it is bounded as maps on the interpolating spaces; and the norm on these interpolation spaces is equivalent to the $\dot{W}^{s,p}(\R^n)$-norm on $C_0^\infty(B_1(\R^n))$ functions, by the same results as for the $k=0$ case.
\end{proof}

\begin{remark}
This lemma could probably be proven, at least for low values of $k$, by brute force evaluation of the Gagliardo seminorm, using the Funk-Hecke theorem (see, e.g., \cite{han2012some}).
\end{remark}

An immediate corollary is the analogous result for vector fields, instead of functions:

\begin{corollary}\label{cor:radial_vectorfield_norm}
Let $n>1$, and define the operator $\tilde{T}:C_c^\infty((0,1))\to C_c^\infty(B_1(\R^n);\R^n)$ by
\[
\tilde{T}f(x) = f(|x|)\frac{x}{|x|}
\]
Then for every $s\ge 0$ and $p\ge 1$, we have
\[
\|\tilde{T}f\|_{W^{s,p}} \le C \|f\|_{W^{s,p}},
\]
for some $C=C(s,p,n)>0$ independent of $f$.
That is, $\tilde{T}:W^{s,p}_0(0,1) \to W^{s,p}_0(B_1(\R^n);\R^n)$ is a bounded operator for any $s\ge 0$ and $p\ge 1$.
\end{corollary}

\begin{proof}
Let $F\in C_c^\infty((0,1))$ be an antiderivative of $f$.
Then the corollary follows from Lemma~\ref{lem:radialfunctionnorm} since $\tilde{T}f = D(TF)$.
\end{proof}

\section{Diameter and displacement energy}
\label{sec:displacement_energy}
In this section we prove a general result relating bounded displacement energy and bounded diameter, inspired by previous results relating zero displacement energy and vanishing geodesic distance \cite{eliashberg1993biinvariant,shelukhin2017hofer,bauer2018vanishing}.
However, as shown below, compared with the vanishing case we need stronger assumptions on the norms involved, assumptions which are too restrictive to the applications in this paper; therefore we used other means to prove boundedness of the diameter.

Let $G$ be a (possibly infinite dimensional) manifold and topological group with neutral element $e$, Lie algebra $\mathfrak g=T_eG$, and left and right translations $L$ and $R$ given by
\begin{equation}
g_1 g_2= L_{g_1}(g_2)=R_{g_2}(g_1),\;\qquad \forall g_1,g_2\in G\;.	
\end{equation}
Assume for each $g\in G$ that $R_g\colon G\to G$ is smooth, and let $\| \cdot \|$ be a norm on the Lie algebra $\mathfrak g$.
This gives rise to the following right-invariant Riemannian metric on $G$:
\begin{equation}
\| h \|_g = \| TR_{g^{-1}}h\|,\qquad \forall g\in G,\; \forall h\in T_g G\;.
\end{equation}
The corresponding geodesic distance function is defined as
\begin{align}
\dist(g_1,g_2)={\operatorname{inf}}   \int_0^1 \| \partial_t g(t)\|_{g(t)} dt\;,\;\qquad \forall g_1,g_2\in G\;,
\end{align}
where the infimum is taken over all smooth paths in $G$ with $g(0)=g_1$ and $g(1)=g_2$.

\begin{theorem}\label{thm:displacement-to-boundedness}
Let $G$ be as above. Assume that
\begin{enumerate}
\item \label{productoftwo} Any transformation $g$ can be written as a product $g=g_1g_2$ where both $g_1$ and $g_2$ are supported on a proper closed subset of $M$. 
\item \label{uniform_perfect} For any proper closed subset $A\subset M$ the group $G_A\subset G$ of all transformations that have support in $A$ is uniformly perfect, i.e., any $g\in G_A$ can be written as a product of $n$ commutators, where $n$ is independent of $g\in G$.
 \item \label{left_translations} The geodesic distance to a commutator of $g$ and $h$ is uniformly controlled by the minimum of the distances to $g$ and $h$, i.e.,
\begin{equation}\label{ass:glob_bounded_left}
\dist(e,[g,h])=\dist(g\circ h,h\circ g) \leq C \operatorname{min}( \dist(e,g),\dist(e,h)),\quad
\forall g,h \in G,
\end{equation}
where $C$ is independent of both $g$ and $h$.\footnote{Note that this holds if the left multiplication $L_g$ is Lipschitz with Lipschitz constant that is independent of $g$, see \cite[Theorem~1]{bauer2018vanishing}.}
\item \label{bounded_displacement} The displacement energy is globally bounded, i.e., for any proper closed subset $A\subset M$ we have
\begin{align}
E(A)=\inf\left\{\dist(e,g):g\in G, g(A)\cap A=\emptyset \right\} \leq D\;
\end{align}
where $D$ is independent of the set $A$.
\end{enumerate}
Then the diameter of the group $G$ is bounded.
\end{theorem}

\begin{proof}
Using Assumption \ref{productoftwo} and the right invariance of the geodesic distance we can reduce 
the boundedness of the diameter to consider only transformations that are supported on a proper closed subset of $M$, since
\begin{multline}
\dist(e,g)=  \dist(e,g_1g_2) =  \dist(g_2^{-1},g_1)\leq \dist(g_2^{-1},e)+\dist(e,g_1)\\=\dist(e,g_2)+\dist(e,g_1)\;,
\end{multline}
where both $g_1$ and $g_2$ are supported in a proper subset of $M$. 

Thus it remains to proof the boundedness of the distance from the identity to any transformation $g$ with support in a proper closed subset $A$.
Using Assumption~\ref{uniform_perfect} we write any $g_1=[h_1,h_2][h_3,h_4]...[h_{2n-1},h_{2n}]$ with $h_i \in G_A$. By the same argument as above we obtain
\begin{equation}
\dist(e,g_1) \leq \sum_{i=1}^n \dist(e,[h_{2i-1},h_{2i-1}])\;.
\end{equation}
To bound the distance from the identity to a commutator of transformations with support in $A$ we proceed as in \cite[Theorem~1]{bauer2018vanishing} and use 
Assumption~\ref{left_translations}
to obtain
\begin{equation}
\dist(e,[h_{2i-1},h_{2i-1}]) \leq (1+C)^2 E(A)\;.
\end{equation}
Putting all of this together we have for each $g\in G$ that
\begin{equation}
 \dist(e,g)\leq 2n(1+C)^2 E(A)
\end{equation}
and using assumption~\ref{bounded_displacement} and the triangle inequality this yields
\begin{equation}
 \dist(g,h)\leq 4 n(1+C)^2 D
\end{equation}
for any $g,h\in G$.
\end{proof}

Let now $M=S^n$ and let $G=\Diff(S^n)$. 
Then Assumptions \ref{productoftwo} and \ref{uniform_perfect} are satisfied \cite{tsuboi2008uniform,burago2008conjugation}.
Assumption \ref{bounded_displacement} is satisfied for $W^{s,p}$-metrics of low enough order, see {Proposition~\ref{pn:displacement_energy_spheres}}. 
In the following we will however show that already in the case $s=1$ and  $n=1$ condition~\ref{left_translations} is to restrictive for our purposes as, e.g., the $\dot H^1$ metric on $\Diff(S^1)$, which corresponds to bounded diameter, does not satisfy it:
\begin{lemma}\label{counter:left_translations_H1}
There exist sequences $\psi_n,\varphi_n \in \operatorname{Diff}(S^1)$ such that 
$\dist_{\dot H^1}(\varphi_n \circ \psi_n,\psi_n\circ \varphi_n) \to \pi/2$ but $\dist_{\dot H^1}(\operatorname{Id},\varphi_n) \to 0$. 
\end{lemma}

\begin{proof}
By the analysis of Lenells \cite{lenells2007hunter}  we have an explicit formula for the geodesic distance of the homogeneous $\dot H^1$-metric given by:
\begin{equation}
 \dist_{1,2}(\psi,\varphi)= \arccos \left( \int_{0}^1 \sqrt{\psi'}\sqrt{\varphi'} d\theta \right)
\end{equation}
Now define the functions 
\begin{equation}
\varphi_n:\;
 \begin{cases} 
      0 &  0 \leq \theta\leq \frac{1}n \\
      2x & \frac1n \leq \theta\leq \frac{2}n \\
      x & \frac2n \leq \theta\leq 1
   \end{cases}\qquad
\psi_n:\;
 \begin{cases} 
      nx &  0 \leq \theta\leq \frac{1}n \\
      1 & \frac1n \leq \theta\leq 1
   \end{cases}   
\end{equation}
The functions $\varphi_n$ and $\psi_n$ are not diffeomorphisms, but we can smooth them with an arbitrarily small change to the $\dot{H}^1$ distances considered.
The claim now follows by a straightforward calculation.
\end{proof}

{\footnotesize
\bibliographystyle{abbrv}

}

\end{document}